\documentclass[a4paper]{amsart}
\usepackage{a4wide}
\usepackage{microtype}
\usepackage{amsmath}
\usepackage{amsthm}
\usepackage{amsfonts}
\usepackage{amssymb}
\usepackage{mathrsfs}
\usepackage{tikz}
\usepackage[bookmarks=false]{hyperref}
\usepackage{graphicx}
\usepackage{overpic}
\usepackage{colortbl}
\usepackage{todonotes}
\usepackage{enumitem}
\usepackage{tabu}
\usepackage{pdflscape}
\usepackage[utf8]{inputenc}

\theoremstyle{plain}
\newtheorem{theorem}{Theorem}[section]
\newtheorem{proposition}[theorem]{Proposition}
\newtheorem{lemma}[theorem]{Lemma}
\newtheorem{corollary}[theorem]{Corollary}

\theoremstyle{definition}
\newtheorem{definition}[theorem]{Definition}
\newtheorem{example}[theorem]{Example}

\theoremstyle{remark}
\newtheorem{remark}[theorem]{Remark}

\DeclareMathOperator{\Span}{span}
\DeclareMathOperator{\PSL}{PSL}

\def \RR {\mathbb R}
\def \CC {\mathbb C}
\def \NN {\mathbb N}

\def \SS {\mathbb S}
\def \HH {\mathbb H}
\def \HP {\mathbb{HP}}
\def \HS {\mathbb{HS}}
\def \dS {\mathrm{d}\mathbb S}
\def \AdS {\mathrm{Ad}\mathbb S}

\def \RRP {\mathbb R \mathrm P}
\def \SSP {\mathbf S}
\def \HHP {\mathbf H}
\def \HPP {\mathbf{HP}}
\def \HSP {\mathbf{HS}}
\def \dSP {\mathrm{d}\mathbf S}
\def \AdSP {\mathrm{Ad}\mathbf S}

\def \cA {\mathscr A}
\def \cP {\mathscr P}

\def \cM {\mathscr M}
\def \cG {\mathscr G}
\def \bx {\mathbf x}
\def \by {\mathbf y}

\begin{document}

\title{Weakly Inscribed Polyhedra}

\author{Hao Chen}
\thanks{Hao Chen was partly supported by the National Science Foundation under Grant
No.\ DMS-1440140 while the author was in residence at the Mathematical Sciences
Research Institute in Berkeley, California, during the Fall 2017 semester.}
\address{Georg-August-Universit\"at G\"ottingen, Institut f\"ur Numerische und Angewandte Mathematik, Lotzestra\ss e 16--18, D-37083 G\"ottingen, Germany}
\email{hao.chen.math@gmail.com}

\author{Jean-Marc Schlenker}
\thanks{Jean-Marc Schlenker's research was partially supported by University of Luxembourg research projects GeoLoDim (9/2014-8/2017) and NeoGeo (9/2017-8/2019), and by the Luxembourg National Research Fund projects INTER/ANR/15/11211745, INTER/ANR/16/11554412/SoS and OPEN/16/11405402.}
\address{Department of Mathematics, University of Luxembourg, 
Maison du nombre, 6 avenue de la Fonte, 
L-4364 Esch-sur-Alzette, Luxembourg.}
\email{jean-marc.schlenker@uni.lu}

\date{\today}

\begin{abstract}
  Motivated by an old question of Steiner, we study convex polyhedra in $\RRP^3$
  with all their vertices on a sphere -- but not necessariy on one side of it -- and
  give an explicit combinatorial description of their possible combinatorics.

  The proof uses a natural extension of the
  3-dimensional hyperbolic space by the de Sitter space. Polyhedra with their
  vertices on the sphere are interpreted as ideal polyhedra in this extended
  space. We characterize the possible dihedral angles of those ideal polyhedra,
  as well as the geometric structures induced on their boundaries, which is
  composed of hyperbolic and de Sitter regions glued along their ideal boundaries.
  
  
\end{abstract}

\maketitle

\tableofcontents

\section{Introduction}
In 1832, Steiner~\cite[Problem 77]{steiner1832} asked the following
questions\footnotemark: Does every polyhedron have a combinatorially equivalent
realization that is inscribed to a sphere, or to another quadratic surface?  If
not, which polyhedra have such realizations?  Here, given a surface $S$, a
polyhedron $P$ is \emph{inscribed to} $S$ if all the vertices of $P$ lie on
$S$.  We say that a polytope is \emph{inscribable to $S$} if it has a
combinatorially equivalent realization with all its vertices on $S$.

\footnotetext{The original text in German is
	\begin{quote}
	77) Wenn irgend ein convexes Polyëder gegeben ist, lässt sich dann immer (oder
	in welchen Fällen nur) irgend ein anderes, welches mit ihm in Hinsicht der Art
	und der Zusammensetzung der Grenzflächen übereinstimmt (oder von gleicher
	Gattung ist), in oder um eine Kugelfläche, oder in oder um irgend eine andere
	Fläche zweiten Grades beschreiben (d.\ h.\ dass seine Ecken alle in dieser
	Fläche liegen oder seine Grenzflächen alle diese Fläche berühren)?
\end{quote}
}

We use the preposition ``to'' rather than ``in'' to make it clear that we do
not require that the polyhedron is on one side of $S$.  This is not required in
Steiner's definition, either.  In fact, since Steiner's problem is obviously
projectively invariant, it is quite natural to consider it in projective space.
It is then possible that a polyhedron with vertices on the sphere does not lie
inside the sphere.

\begin{definition}\label{def:strongweak}
  In the projective space $\RRP^3$, a polyhedron $P$ inscribed to a quadric
  $S$ is \emph{strongly inscribed to $S$} if the interior of $P$ is
  disjoint from $S$, or \emph{weakly inscribed to $S$} otherwise.
\end{definition}

Steiner also defined that a polyhedron is \emph{circumscribed} to a surface if
all its facets are tangent to the surface.  We will see that circumscription
and inscription are closely related through polarity, hence we only need to
focus on one of them.

\medskip

Steiner's problem remained entirely open for nearly a century.  There were even
beliefs~\cite{bruckner1900} that every simplicial polyhedron is strongly
inscribable in a sphere.  The first polyhedra without any strongly inscribed
realization were discovered in 1928 by Steinitz~\cite{steinitz1928}.  It was
realized much later that the cube with one vertex truncated cannot be inscribed
to any quadric. This follows from the well-known fact that if seven vertices of
a cube lie on a quadric, so does the eighth
one~\cite[Section~3.2]{bobenko2008}; see~\cite[Example~4.1]{chen2017} for a
complete argument.

There are three quadrics in $\RRP^3$ up to projective transformation: the
sphere, the one-sheeted hyperboloid, and the cylinder.  Strong inscriptions to
them are essentially characterized in previous works of
Hodgson--Rivin--Smith~\cite{hodgson1992} and
Danciger--Maloni--Schlenker~\cite{danciger2014}.  The current paper answers
Steiner's question for polyhedra weakly inscribed to a sphere.

\medskip

The projective space $\RRP^3$ can be seen as a completion of the Euclidean
space $\RR^3$ with a hyperplane at infinity.  If the hyperplane at infinity is
disjoint from both the sphere and the inscribed polyhedron, then the
inscription must be strong (see~\cite{chen2017}).  Hence we focus on the
situation where the hyperplane at infinity intersects the sphere, but does not
intersect the polyhedron.  The sphere then appears in the Euclidean space as a
two-sheeted hyperboloid, and the weakly inscribed polyhedron has some vertices
on one sheet, and other vertices on the other sheet.

Our main result is the following combinatorial characterization of polyhedra
weakly inscribed to a sphere.

\begin{theorem}[Combinatorial characterization] \label{thm:combinatorics}
  A $3$-connected planar graph $\Gamma=(V,E)$ is the $1$-skeleton of a
  polyhedron $P \subset \RRP^3$ weakly inscribed to a sphere if and only if
  \renewcommand{\theenumi}{\rm\bf(C\arabic{enumi})}
  \renewcommand{\labelenumi}{\theenumi}
  \begin{enumerate}
  	\item \label{con:cycles} $\Gamma$ admits a vertex-disjoint cycle cover
  		consisting of two cycles
  \end{enumerate}
  and, if we color the edges connecting vertices on the same cycle by red (r), and
  those connecting vertices from different cycles by blue (b), then
  \begin{enumerate}[resume]
		\item \label{con:alternate} there is a cycle visiting all the edges
			(repetition allowed) along which the edge color has the pattern
			\begin{itemize}
				\item \dots bbrbbr\dots~if the cycle cover contains a $1$-cycle, or

				\item \dots brbr\dots~otherwise.
			\end{itemize}
	\end{enumerate}
\end{theorem}

Here, we abuse the terminology, and call a single vertex $1$-cycle,
and a single edge $2$-cycle.  We will see that the two cycles
in~\ref{con:cycles} correspond to vertices on the two sheets of the
hyperboloid, and the edges are colored blue if they are between the sheets, or
red otherwise.  If the cycle cover contains a $1$-cycle with a single vertex
$v$, Condition~\ref{con:alternate} has a much simpler formulation, namely that
$v$ is connected to every other vertex.  We decide to adopt the current
formulation for comparison with the other case.

Theorem~\ref{thm:combinatorics} is remarkable because, unlike the
characterization of strongly inscribed polyhedra~\cite{hodgson1992}, it does
not involve any feasibility problem.  Despite some
efforts~\cite{dillencourt1996, cheung2003}, no characterization as explicit as
Theorem~\ref{thm:combinatorics} has been obtained for strong inscription.

In fact, Theorem~\ref{thm:combinatorics} is a consequence of the following
linear programming characterization.

\begin{theorem}[Linear programming characterization] \label{thm:weights}
  A $3$-connected planar graph $\Gamma=(V,E)$ is the $1$-skeleton of a
  polyhedron $P \subset \RRP^3$ weakly inscribed to a sphere if and only if
  \renewcommand{\theenumi}{\rm\bf(C\arabic{enumi})}
  \renewcommand{\labelenumi}{\theenumi}
  \begin{enumerate}
  	\item $\Gamma$ admits a vertex-disjoint cycle cover consisting of two
  		cycles
  \end{enumerate}

  and, if we color the edges connecting vertices on the same cycle by red, and
  those connecting vertices from different cycles by blue, there is a weight
  function $w:E\to\RR$ such that

	\renewcommand{\theenumi}{\rm\bf(W\arabic{enumi})}
	\renewcommand{\labelenumi}{\theenumi}
	\begin{enumerate}
  	\item \label{con:sign} $w>0$ on red edges, and $w<0$ on blue edges;

  	\item \label{con:wsum} $w$ sums up to $0$ over the edges adjacent to a
  		vertex $v$, except when $v$ is the only vertex in a $1$-cycle, in which
  		case $w$ sums up to $-2\pi$ over the edges adjacent to $v$. 
  \end{enumerate}
\end{theorem}

Recall that we consider a single vertex as a $1$-cycle, and a single edge as a
$2$-cycle.  The feasibility problem involved here is visibly much simpler than
that in~\cite{hodgson1992}.  In particular, there is no inequality on the
non-trivial cuts.

\medskip

In~\cite{hodgson1992}, the sphere was seen as the ideal boundary of the
projective model of the hyperbolic space, so that strongly inscribed polyhedra
were interpreted as hyperbolic ideal polyhedra.  The characterization is then
formulated in terms of hyperbolic exterior dihedral angles.  More specifically,
a polyhedron is strongly inscribable in a sphere if and only if one can assign
weights (angles) to the edges subject to a family of equalities and
inequalities.  Hence strong inscribability in the sphere can be determined by
solving a feasibility problem.

In~\cite{danciger2014}, polyhedra strongly inscribed in a one-sheeted
hyperboloid (resp.\ a cylinder) were seen as ideal polyhedra in the anti-de
Sitter space (resp.\ the half-pipe space~\cite{danciger:ideal,
danciger:transition}).  Dihedral angles of the ideal polyhedra then lead to a
linear programming characterization in the same style of~\cite{hodgson1992}.

\begin{remark}
  The definition in~\cite{danciger2014} is slightly stronger than ours.  They
  require that $P \cap S$ consists of exactly the vertices of $P$.  The two
  definitions are equivalent only when $S$ is the sphere.  Otherwise, it is
  possible that some edges of $P$ are contained in $S$.
\end{remark}

Our proof to Theorem~\ref{thm:weights} follows a similar approach.  Given a
sphere $S\subset\RRP^3$, its interior (resp.\ exterior) is seen as the
projective model for the $3$-dimensional hyperbolic space $\HHP^3$ (resp.\ de
Sitter space $\dSP^3$).  In~\cite{schlenker1998, schlenker2001}, $\HHP^3$ and
$\dSP^3$ together make up the \emph{hyperbolic-de Sitter space} (HS space for
short) which is denoted by $\HSP^3$.  Then a polyhedron $P$ ideal to $S$ can be
considered as an ideal polyhedron in $\HSP^3$.  We say that $P$ is strongly
ideal if $P$ is contained in $\HHP^3$, or weakly ideal otherwise.

We will see in Section \ref{sec:combinatorics} that Theorem \ref{thm:weights}
follows from Theorem \ref{thm:angle} below, which describes the possible
dihedral angles of convex polyhedra in $\HSP^3$.  More specifically, the
dihedral angles at the edges of $P$ form a weight function $\theta$ satisfying
all the conditions of Theorem~\ref{thm:weights} and, additionally, that
$|\theta|<\pi$ and the sum of $\theta$ over the blue edges is bigger than
$-2\pi$.  These additional conditions are, however, redundant in the linear
programming characterization, as we will prove in
Section~\ref{sec:combinatorics}.

\medskip

Rivin~\cite{rivin1994} also gave another characterization in terms of the
metric induced on the boundaries of ideal hyperbolic polyhedra.  More
specifically, every complete hyperbolic metric of finite area on an $n$-times
punctured sphere can be isometrically embedded as the boundary of a
$n$-vertices polyhedron strongly inscribed to a sphere, viewed as an ideal
hyperbolic polyhedron (possibly degenerate and contained in a plane).
Similarly, \cite{danciger2014} also characterized polyhedra strongly inscribed
in the one-sheeted hyperboloid in termes the possible induced metrics on the
boundary of ideal polyhedra in the Anti-de Sitter space.

Extending these previous works, we also provide a characterization for the
geometric structure induced on the boundary of a weakly ideal polyhedron in
$\HSP^3$.  This geometric structure, as distinguished from that induced on a
strongly ideal polyhedron, contains a de Sitter part; that is, a part locally
modeled on the de Sitter plane.  We call this induced data an ``HS-structure'',
since it is locally modeled on $\HS^2$, a natural extension of the hyperbolic
plane by the de Sitter plane. Relevant definitions in the following statement
will be recalled in the next section.

\begin{theorem}[Metric characterization] \label{thm:induced}
  Let $P$ be a weakly ideal polyhedron in $\HSP^3$ with $n$ vertices. Then the
  induced HS-structure on $\partial P$ is a complete, maximal HS structure on
  the punctured sphere, obtained by gluing copies of $\HHP^2$ to a de Sitter
  surface along their ideal boundaries by $C^1$ piecewise projective maps such
  that, at the ``break points'' where the maps fail to be projective, the
  second derivative has a positive jump.  Conversely, each HS structure of this
  type is induced on a unique weakly ideal polyhedron in $\HSP^3$. 
\end{theorem}

Note that both the hyperbolic and de Sitter parts of the metric have a
well-defined real projective structure at infinity, so it is meaningful to ask
for a piecewise projective gluing map.  More explanations on the statement of
Theorem \ref{thm:induced} can be found in Section \ref{ssc:HS}.

\begin{remark}
	For the interest of physics audience, weakly ideal polyhedra can be
	interpreted as a description of interactions of ``photons'' in a 3-dimensional spacetime;
	see~\cite{barbot2011}.  More specifically, $\HS^2$ models the link of an
	event in a 3-dimensional space-time.  The vertices on different boundary components of
	the de Sitter surface in Theorem~\ref{thm:induced}, or, combinatorially, on
	different cycles in Condition~\ref{con:cycles}, correspond to incoming and
	outgoing photons (depending on the direction of time) involved in an
	interaction.  A special case is the single vertex in a degenerate boundary
	component, or, combinatorially, in a $1$-cycle, which corresponds to an
	extreme BTZ-like singularity.\footnote{Changed this paragraph for clarity}
\end{remark}

\begin{remark}
  The weak inscription, although covered by Steiner's definition, seemed
  forgotten and only revived recently.  Schulte~\cite{schulte1987} considered
  higher dimensional generalizations of Steiner's problem and defined a weaker
  notion following an idea from~\cite{grunbaum1987}.  However, since he worked
  in Euclidean space, his definition coincides with the strong inscription.
  Padrol and the first author~\cite{chen2017} extended Schulte's definitions
  into the projective space, and noticed polyhedra inscribed to the sphere but
  not strongly inscribable.
\end{remark}

\medskip

The paper is organized as follows.  The essential definitions are made in
Section~\ref{sec:define} in a general setting.  Then we can view polyhedra
weakly inscribed to the sphere as ideal polyhedra in $\HSP^3$.  In
Section~\ref{sec:overview}, we announce characterizations for the dihedral
angles and induced metrics of weakly ideal polyhedra in $\HSP^3$, which are
actually reformulations of Theorems~\ref{thm:weights} and~\ref{thm:induced}.
We also outline the proof strategy, which is carried out in the following
sections.  In particular, there are some technical challenges which were not
encountered in the previous works on strong inscription.  For instance, the
space of weakly inscribed polyhedra is not simply connected.  Finally, in
Section~\ref{sec:combinatorics}, we deduce Theorem~\ref{thm:combinatorics} from
the linear programming characterization of dihedral angles.

\section{Definitions} \label{sec:define}
We are mainly interested in three dimensional polyhedra.  However, the
definitions in this section are more general than strictly necessary, and cover
the anti-de Sitter and half-pipe spaces that we hope to study in a further
work. Lower dimensional cases are used as examples.

\subsection{The hyperbolic, anti-de Sitter and half-pipe spaces}

The \emph{projective space} $\RRP^d$ is the set of linear $1$-subspaces of
$\RR^{d+1}$.  An \emph{affine chart} of $\RRP^d$ is an affine hyperplane $H
\subset \RR^{d+1}$ which is identified to the set of linear $1$-dimensional
subspaces intersecting $H$.  The linear hyperplane parallel to $H$ is
projectivized as the \emph{hyperplane at infinity}; linear $1$-dimensional
subspaces contained in this hyperplane have no representation in the affine
chart $H$.

Let $\RR^{d+1}_{p,q}$ denotes $\RR^{d+1}$ equipped with an inner product
$\left< \cdot, \cdot \right>$ of signature $(p,q)$, $p+q \le d+1$.  We say that
$\RR^{d+1}_{p,q}$ is \emph{non-degenerate} if $p+q = d+1$.  For convenience, we
will assume that
\[
	\left< \bx, \by \right> := \sum_{i=0}^{p-1} x_iy_i - \sum_{i=0}^{q-1}
	x_{d-i}y_{d-i}.
\]

For $p+q \le d+1$, we define $\HH^d_{p,q}=\{\bx \in \RR^{d+1}_{p,q} \mid \left<
\bx, \bx \right> = -1\}$ and $\SS^d_{p,q} = \{\bx \in \RR^{d+1}_{p,q}
\mid \left< \bx, \bx \right> = 1 \}$.  Both $\HH^d_{p,q}$ and $\SS^d_{p,q}$ are
equipped with the metric induced by the inner product.  The metrics of
$\HH^d_{p,q}$ and $\SS^d_{q,p}$ differ only by a sign.

Here and through out this paper, if a space in $\RR^{d+1}_{p,q}$ is denoted by
a blackboard boldface letter, we use the corresponding simple boldface letter
to denote its projectivization in $\RRP^d_{p,q}$.  For example, $\HHP^d_{p,q}$
and $\SSP^d_{p,q}$ are the quotient of $\HH^d_{p,q}$ and $\SS^d_{p,q}$ by the
antipodal map.

\begin{example}\leavevmode
	\begin{itemize}
		\item $\HHP^d := \HHP^d_{d,1}$ is the projective model of the hyperbolic space;

		\item $\HP^d := \HH^d_{d-1,1}$ is the half-pipe space, see \cite{danciger:transition};

		\item $\AdS^d := \HH^d_{d-1,2}$ is the Anti-de Sitter space;

		\item $\SS^d := \SS^d_{d+1,0}$ is the spherical space;

		\item $\dS^d := \SS^d_{d,1}$ is the de Sitter space.
	\end{itemize}
\end{example}

We define $\HS^d_{p,q} = \{ \bx \in \RR^{d+1}_{p,q} \mid |\langle \bx, \bx
\rangle| = 1 \}$.  It is equipped with a complex-valued ``distance'', which
restricts to each connected component as the natural constant curvature metric,
and can be defined in terms of the Hilbert metric of the boundary quadric,
see~\cite{schlenker1998}.  If $p$ and $q$ are both non-zero, $\HS^d_{p,q}$
consists of a copy of $\HH^d_{p,q}$ and a copy of $\SS^d_{p,q}$ identified
along their ideal boundaries.

\begin{example}\leavevmode
	\begin{itemize}
		\item $\HS^d_{d,1}$ consists of two copies of the hyperbolic space $\HHP^d$
			and a copy of the de Sitter space $\dS^d$; we call it the ``hyperbolic-de
			Sitter space'', and simplify the notation to $\HS^d$.

		\item Another situation that concerns us in the future is $\HS^d_{p,p}$,
			$2p<d$ consisting of two copies of $\HH^d_{p,p}$ differing by the sign of
			the metric. We denote it by $2\HH^d_{p,p}$.  In particular,
			$\HS^3_{2,2}=2\AdS^3$.

		\item Up to a sign of metric, there are five possible $\HS^2_{p,q}$
			metrics, namely $\HS^2_{3,0}(= \SS^2)$, $\HS^2_{2,1}(= \HS^2)$,
			$\HS^2_{2,0}$, $\HS^2_{1,1}(= 2\HP^2)$ and $\HS^2_{1,0}$.

		\item Up to a sign of metric, there are three possible $\HS^1_{p,q}$
			metrics.  We call a $1$-subspace \emph{space-}, \emph{light-} or
			\emph{time-}like if it is isometric to $\HS^1_{1,1}$, $\HS^1_{1,0}$ or
			$\HS^1_{2,0}$ respectively.
	\end{itemize}
\end{example}

In an affine chart of $\RRP^d_{p,q}$, the boundary $\partial \HSP^d_{p,q} :=
\partial \HHP^d_{p,q} = \partial \SSP^d_{p,q}$ appears as a quadric in
$\RR^d$.

\begin{example}\leavevmode
	\begin{itemize}
		\item In the affine chart $x_3=1$: $\HHP^3$ appears as a unit open ball;
			$\HPP^3$ appears as the interior of a circular cylinder; $\AdSP^3$ appears
			as the simply connected side of a one-sheeted hyperboloid.

		\item In the affine chart $x_2=1$: $\HHP^3$ appears as the two components of
			the complement of a two-sheeted hyperboloid that do not share a boundary;
			$\HPP^3$ appears as two circular cones; $\AdSP^3$ appears as the non-simply
			connected side of a one-sheeted hyperboloid.
	\end{itemize}
\end{example}

A totally geodesic subspace in $\HSP^d_{p,q}$ is given by a projective subspace
$L \subset \RRP^d$.  If $L$ is of codimension $k$, then the induced metric on
$L$ is isometric to $\HSP^d_{p',q'}$ for some $p'+q' \le d-k$ and, by Cauchy's
interlacing theorem, we have $0 \le p-p' \le k$ and $0 \le q-q' \le k$.  If
$\HSP^d_{p,q}$ is non-degenerate, then there are three possible metrics on a
totally geodesic hyperplane $H$ (codimension $1$).  We say that $H$ is
\emph{space-}, \emph{time-} or \emph{light-like} if it is isometric to
$\HSP^{d-1}_{p,q-1}$, $\HSP^{d-1}_{p-1,q}$ or $\HSP^{d-1}_{p-1,q-1}$,
respectively.  

\begin{example}\leavevmode
	In $\HSP^d$, a hyperplane $H$ is space-like if it is disjoint from the closure
	of $\HHP^d$, time-like if it intersects $\HHP^d$, or light-like if it is
	tangent to the boundary of $\HHP^d$.
\end{example}

The polar of a set $X \subset \RRP^d$ is defined by
\[
	X^* = \{ \left[\bx\right] \colon \left< \bx, \by \right> \le 0 \text{ for all
	} \left[\by\right] \in X \}.
\]
The polar of a subspace $L \subset \RRP^d$ is its orthogonal companion, i.e.\
\[
	L^* = L^\perp = \{ \left[\bx\right] \colon \left< \bx, \by \right> = 0 \text{
	for all } \left[\by\right] \in L \}.
\]
If $\HSP^d_{p,q}$ is non-degenerate, $L$ is isometric to $\HSP^k_{r,s}$ and
$L^\perp$ to $\HSP^{k'}_{r',s'}$, then we have $k+k' = d-1$ and $k-r-s =
k'-r'-s' = p-r-r' = q-s-s'$.  In particular, the polar of a hyperplane $H$ is a
point in $\HHP^d_{p,q}$ if $H$ is space-like, in $\SSP^d_{q,p}$ if $H$ is
time-like, or on $\partial \HSP^d_{p,q}$ if $H$ is light-like.


\subsection{Ideal polytopes}

A set $X \subset \RRP^d$ is \emph{convex} if it is convex in some affine chart
that contains it.  Equivalently~\cite{degroot1958}, $X \subset \RRP^d$ is
convex if for any two points $p,q \in X$, exactly one of the two segments
joining $p$ and $q$ is contained in $X$.
\begin{example}\leavevmode
	\begin{itemize}
		\item $\HHP^d$ is convex;

		\item $\AdSP^d$ is not convex;

		\item $\HPP^d$ is convex, but its closure is not.
	\end{itemize}
\end{example}
Two convex sets are \emph{consistent} if some affine chart contains both of
them, or \emph{inconsistent} otherwise.

A \emph{convex hull} of a set $X$ is a minimal convex set containing $X$.  Note
that there is usually more than one convex hull.  A \emph{convex polytope} $P$
is a convex hull of finitely many points.  A (closed) \emph{face} of $P$ is the
intersection of $P$ with a \emph{supporting hyperplane}, i.e.\ a hyperplane
that intersects the boundary of $P$ but disjoint from the interior of $P$.  The
faces of $P$ decompose the boundary $\partial P$ into a cell complex, giving a
face lattice.  Two polytopes are \emph{combinatorially equivalent} if they have
the same face lattice.  The polar $P^*$ is combinatorially dual to $P$, i.e.\
the face lattice of $P^*$ is obtained from $P$ by reversing the inclusion
relations.   We recommend the books~\cite{grunbaum2003, ziegler1995} as general
references for polytope theory.

%
%
%

\begin{definition}
	A convex polytope $P \subset \RRP^d$ is \emph{ideal} to $\HSP^d_{p,q}$ if
	all its vertices are on the boundary of $\HSP^d_{p,q}$.  An ideal polytope
	$P$ is \emph{strongly ideal} if the interior of $P$ is disjoint from
	$\partial\HSP^d_{p,q}$, or \emph{weakly ideal} otherwise.
\end{definition}

In the case that $P$ is (strongly) ideal to $\HSP^d_{p,q}$, we also say that
$P$ is (strongly) ideal to $\HHP^d_{p,q}$ or to $\SSP^d_{p,q}$.

A \emph{(polyhedral) HS structure} of a $k$-dimensional manifold is a
triangulation the manifold together with an isometric embedding of each
$k$-simplex into $\HSP^k_{p,q}$, such that the simplices are isometrically
identified on their common faces.  If $\RRP^d$ is equipped with a
$\HSP^d_{p,q}$ metric, then a convex polytope $P \subset \RRP^d$ with $n$
vertices naturally induces an HS structure on the $n$-times punctured
$\SS^{d-1}$ .  If $P$ is ideal to $\HSP^d_{p,q}$, then this metric is
geodesically complete.

In any affine chart, $\partial \HSP^d_{p,q}$ appears as a quadratic surface,
and an ideal polytope appears inscribed to this surface.

For $\HHP^d$, an ideal polytope $P$ is strongly ideal if and only if it is
consistent with $\HHP^d$~\cite{chen2017}.  Polyhedra strongly ideal to $\HHP^3$
are then inscribed to a sphere.  Their combinatorics was characterized by
Hodgson, Rivin and Smith \cite{hodgson1992}.  Polyhedra strongly ideal to
$\HPP^3$ are inscribed to a circular cylinder.  Polyhedra strongly ideal to
$\AdSP^3$ are inscribed to and contained in a one-sheeted hyperboloid.
Danciger, Maloni and the second author~\cite{danciger2014} have
\emph{essentially} provided characterizations of the combinatoric types of
these polytopes.

We will focus on \emph{weakly ideal} polyhedra, i.e.\ ideal polyhedra that are
not strongly ideal.  We prefer affine charts that contains the polytope $P$;
such an affine charts cannot contain $\HHP^3$, $\HPP^3$ or $\AdSP^3$ by the
discussion above.  Polyhedra weakly ideal to $\HHP^3$ are then inscribed to a
two-sheeted hyperboloid.  Polyhedra weakly ideal to $\HPP^3$ are inscribed to a
circular cone.  And finally, polyhedra weakly ideal to $\AdSP^3$ are inscribed
to, but not contained in, a one-sheeted hyperboloid.  This covers all the
quadratic surfaces, and characterizing the weakly ideal polyhedra in $\HHP^3,
\HPP^3$ and $\AdSP^3$ would provides a complete answer to Steiner's problem.

\section{Overview} \label{sec:overview}
From now on, we will focus on projective polyhedra weakly inscribed to the
sphere, which is equivalent to projective polyhedra weakly ideal to $\HHP^3$,
or Euclidean polyhedra inscribed to the two-sheeted hyperboloid.

Recall that a polyhedron $P$ weakly ideal to $\HHP^3$ is not consistent with
$\HHP^3$.  Since we prefer affine charts containing $P$, $\HHP^3$ would appear,
up to a projective transformation, as the set $x_0^2+x_1^2-x_2^2 < -1$ in such
charts.  This is projectively equivalent to the Klein model.  We use $\HHP^3_+$
and $\HHP^3_-$ to denote the parts of $\HHP^3$ with $x_2>0$ and $x_2<0$,
respectively.  Moreover, the boundary $\partial \HHP^3$ appears as a two-sheeted
hyperboloid.

\subsection{Ideal polyhedra}

For a polyhedron $P$ weakly ideal to $\HHP^3$, let $V$ denotes the set of its
vertices; then $V \subset \partial \HHP^3$ by definition.  We write $V^+=V \cap
\partial\HHP^3_+$ and $V^-=V \cap \partial\HHP^3_-$, and say that $P$ is
$(p,q)$-ideal if $|V^+|=p$ and $|V^-|=q$.  $P$ is strongly ideal if $p=0$ or
$q=0$; we only consider weakly ideal polyhedra, hence $p>0$ and $q>0$.
Following the curves $P \cap \partial \HHP^3$, we label vertices of $V^+$ by
$1^+, \dots, p^+$, and vertices of $V^-$ by $1^-, \dots, q^-$, in the order
compatible with the right-hand rule.

Let $\cP_n$ denote the space of labeled polyhedra with $n \ge 4$ vertices that
are weakly ideal to $\HHP^3$, considered up to hyperbolic isometries, and
$\cP_{p,q}$ denote the space of labeled $(p,q)$-ideal polyhedra, $p + q \ge 4$.
Then $\cP_n$ is the disjoint union of $\cP_{p,q}$ with $p+q=n$.  We only need
to study connected components $\cP_{p,q}$, and may assume $p \le q$ without
loss of generality.  We usually distinguish two cases, namely $p<2<q$ and $2
\le p \le q$.  Note that we always assume that $p\geq 1$ since we only consider
only {\em weakly} ideal polyhedra, so $p<2$ below always means $p=1$.

\subsection{Admissible graphs}

We define a \emph{weighted graph} (or simply \emph{graph}) on a set of vertices
$V$ as a real valued function $w$ defined on the unordered pairs
$\binom{|V|}{2}$.  The \emph{weight} $w_v$ at a vertex $v \in V$ is defined as
the sum $\sum_u \theta(u,v)$ over all $u \ne v$.

Unless stated otherwise, the support of $w$ is understood as the set $E$ of
edges.  We can treat $w$ as a usual graph with edge weight $\theta$, and talk
about notions such as subgraph, planarity and connectedness.  But we will also
take the liberty to include edges of zero weight, as long as it does not
destroy the property in the center of our interest.  For example, graphs in
this paper are used to describe the $1$-skeleta of polyhedra, i.e.\
$3$-connected planar graphs.  Hence whenever convenient, we will consider maximal
planar triangulations.  If this is not the case with the support of $\theta$,
we just triangulate the non-triangle faces by including edges of zero weight.

The advantage of this unconventional definition is that graphs can be treated
as vectors in $\RR^{\binom{|V|}{2}}$.  Weighted graphs of a fixed
combinatorics, together with their subgraphs, then form a linear subspace.
Graphs with a common subgraph correspond to subspaces with nontrivial
intersection.  This makes it convenient to talk about neighborhood,
convergence, etc.  For a fixed polyhedral combinatorics, our main result
implies that the set of weighted graphs form a $(|E|-|V|)$-dimensional cell.
Weighted graphs of a fixed number of vertices then form a cell complex of
dimension $2|V|-6$ in $\RR^{\binom{|V|}{2}}$: The maximal cells correspond to
triangulated (maximal) planar graphs, and they are glued along their faces
corresponding to common subgraphs.

\medskip

Consider an edge $e$ of an ideal polyhedron $P$.  Then $e$ is either a geodesic
in $\HHP^3$, or a time-like geodesic in $\dSP^3$.  In both cases, the faces
bounded by $e$ expand to half-planes forming a hyperbolic exterior dihedral
angle, denoted by $\vartheta$.  We assign to $e$ the \emph{HS exterior dihedral
angle} $\theta$, which equals $\vartheta$ if $e \subset \HHP^3$, or
$-\vartheta$ if $e \subset \dSP^3$.  We will refer to $\theta$ as exterior
angles, dihedral angles, or simply angles, and should not cause any confusion.  

This angle assignment induces a graph on $V$, also denoted by $\theta$,
supported by the edges of $P$.  We have thus obtained a function $\Theta$ that
maps an ideal polyhedron $P$ to the graph $\theta$ of its angles.  Obviously,
$\Theta(P)$ is polyhedral, i.e.\ $3$-connected planar.  We will see that, if
$P$ is $(p,q)$-ideal, then

\renewcommand{\theenumi}{\rm\bf(C\arabic{enumi})}
\renewcommand{\labelenumi}{\theenumi}
\begin{enumerate}
	\item \label{con:2cycles} $\theta=\Theta(P)$ admits a vertex-disjoint cycle
		cover consisting of a $p$-cycle and a $q$-cycle
\end{enumerate}
and, if we color the edges connecting vertices on the same cycle by red, and
those connecting vertices from different cycles by blue, then
\renewcommand{\theenumi}{\rm\bf(A\arabic{enumi})}
\renewcommand{\labelenumi}{\theenumi}
\begin{enumerate}
	\item \label{con:range} $0 < \theta < \pi$ on red edges, and $-\pi < \theta <
		0$ on blue edges;

	\item \label{con:vertex} $\theta_v=\sum_u \theta(u,v)=0$, with the exception
		when $v$ is the only vertex in a $1$-cycle, in which case $\theta_v =
		-2\pi$;

  \item \label{con:negsum} The sum of $\theta$ over blue edges is $\le -2\pi$,
  	and the equality only happens when $p < 2 < q$.
\end{enumerate}

The exception in Condition~\ref{con:vertex} only happens when $p < 2 < q$.

\begin{definition} \label{def:pq}
	A \emph{$(p,q)$-admissible} graph is a weighted polyhedral graph satisfying
	Conditions~\ref{con:2cycles} and \ref{con:range}--\ref{con:negsum}.
\end{definition}

Given a $(p,q)$-admissible graph drawn on the plane, we may label the vertices
on the $p$-cycle by $1^+, \dots, p^+$, and vertices on the $q$-cycle by $1^-,
\dots, q^-$, both in the clockwise order.  Let $\cA_{p,q}$ denote the space of
labeled $(p,q)$-admissible graphs with $p + q \ge 4$.  We use $\cA_n$, $n \ge
4$, to denote the disjoint union of $\cA_{p,q}$ with $p+q=n$.  Our main
Theorem~\ref{thm:weights} is the consequence of the following theorem:

\begin{theorem}\label{thm:angle}
	$\Theta$ is a homeomorphism from $\cP_{p,q}$ to $\cA_{p,q}$.
\end{theorem}

Note that Condition~\ref{con:negsum} and some inequalities in
Condition~\ref{con:range} are not present in Theorem~\ref{thm:weights}.  We will
see that they are indeed redundant when formulating a feasibility problem.

\subsection{Admissible HS structures}
\label{ssc:HS}

Let $\Delta$ denote the function that maps an ideal polyhedron to its induced
HS structure.  If $P$ is $(p,q)$-ideal, it follows from the definition that
$\Delta(P)$ is geodesically complete, and it is maximal in the sense that it
does not embed isometrically as a proper subset of another HS structure.

The part of $\partial P$ in $\HHP^3$ has no interior vertex, hence is isometric
to a disjoint union of copies of $\HHP^2$.  In the case $2 \le p \le q$, we use
$\HHP^2_\pm$ to denote the copies induced by $\partial P \cap \HHP^3_\pm$.  If
$p < 2 < q$, we have only $\HHP^2_- = \partial P \cap \HHP^3_-$.  The part of
$\partial P$ in $\dSP^3$ has no interior vertex, neither, hence $\partial P \cap
\dSP^3$ is isometric to a complete de Sitter surface.

The intersection of $\partial P$ with a space-like plane in $\dSP^3$ is a
simple polygonal closed space-like curve in $\partial P \cap \dSP^3$. If $2 \le
p \le q$, this polygonal curve can be deformed to one of maximal length, say
$\gamma_0$, which is therefore geodesic in $\partial P \cap \dSP^3$. Considered
as a polygonal curve in $\dSP^3$, $\gamma_0$ is then E-convex in the sense of
\cite[Def 7.13]{schlenker1998}, and it follows that its length $\ell$ is less
than $2\pi$, see \cite[Prop 7.14]{schlenker1998}. As a consequence, $\gamma_0$
is the unique simple closed space-like geodesic in $\partial P \cap \dSP^3$,
because any other simple closed space-like geodesic would need to cross
$\gamma_0$ at least twice (there is no de Sitter annulus with space-like,
geodesic boundary by the Gauss-Bonnet formula), and two successive intersection
points would be separated by a distance $\pi$, leading to a contradiction.  We
denote the metric space $\partial P \cap \dSP^3$ by $\dS^2_\ell$.  $\dS^2_\ell$
has two boundary components, both homeomorphic to a circle.

If $p<2<q$, then one boundary component of $S$ degenerates to a point.  In this
case, the metric space $\partial P \cap \dSP^3$ does not contain any closed
space-like geodesic, and we denote it by $\dS^2_0$.

Hence $\Delta(P)$ is obtained by gluing one or two copies of $\HHP^2$ to the
non-degenerate boundary components of a de Sitter surface.  Let $\gamma_\pm$ be
the map that glues $\partial \HHP^2_\pm$ to $\partial \dS^2_\ell$.  We will see
that $\gamma_\pm$ are $C^1$ piecewise projective maps (CPP maps for short).
More specifically, they are projective except at the vertices of $P$.  The
points where the map is not projective are called \emph{break points}.  A break
point is said to be positive (resp.\ negative) if the jump in the second
derivative at this point is positive (resp.\ negative).  We will see that the
break points of $\gamma_\pm$ are all positive.

\begin{definition}
	A \emph{$(p,q)$-admissible} HS structure, $p+q \ge 4$, is obtained

	\begin{description}
		\item[In the case $p<2<q$] by gluing a copy of $\HHP^2$ to $\dS^2_0$ along the
			non-degenerate ideal boundary by a CPP map with $q$ positive break
			points.
	
		\item[In the case $2 \le p \le q$] by gluing two copies of $\HHP^2$ to
			$\dS^2_\ell$, $0 < \ell < 2\pi$, along the ideal boundaries through CPP
			maps with, respectively, $p$ and $q$ positive break points.
	\end{description}
\end{definition}

In Figure~\ref{fig:metric} we sketch the situation of $p=2$ and $q=3$ is
sketched in Figure~\ref{fig:metric}.

\begin{figure}[hbt] 
  \centering 
  \includegraphics[width=.2\textwidth]{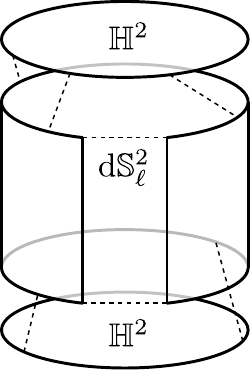}
 	\caption{
 		A sketch of a $(2,3)$-admissible HS structure.  The dashed segments
 		indicates the gluing maps, including one that produces $\dS^2_\ell$, and
 		two CPP maps with two and three break points, respectively.
 		\label{fig:metric}
	}
\end{figure}

Given a $(p,q)$-admissible HS structure, we may label the break points in the
two boundary components of $S$ by $1^+, \dots, p^+$ and $1^-, \dots, q^-$,
respectively.  Let $\cM_{p,q}$ denote the space of $(p,q)$-ideal HS structures
up to isometries.  Our main Theorem~\ref{thm:induced} is the consequence of the
following theorem.

\begin{theorem}\label{thm:metric}
	$\Delta$ is a homeomorphism from $\cP_{p,q}$ to $\cM_{p,q}$.
\end{theorem}

\subsection{Outline of proofs}

We will prove that $\Theta$ and $\Delta$ are local immersions
(Section~\ref{sec:rigidity}) with images in $\cA_{p,q}$
(Section~\ref{sec:extrinsic}) and $\cM_{p,q}$ (Section~\ref{sec:intrinsic}),
respectively.  They are then local homeomorphisms because $\cP_{p,q}$,
$\cA_{p,q}$ and $\cM_{p,q}$ have the same dimension $2(p+q-3)$
(Section~\ref{sec:topology}).  Moreover, they are proper maps
(Section~\ref{sec:properness}), hence are covering maps.  A difference from the
previous works lies in the fact that $\cP_{p,q}$, $\cA_{p,q}$ and $\cM_{p,q}$
are not simply connected if $2 \le p \le q$.  We will use open covers and
universal covers to conclude that the covering numbers of $\Theta$ and $\Delta$
are one (Section~\ref{sec:topology}).

\section{Necessity}\label{sec:necessity}
\subsection{Combinatorial conditions}

We first verify combinatorial Condition~\ref{con:2cycles}.  For this we will
need some lemmata about convex sets in $\RRP^d$.

\begin{lemma}\label{lem:convcomp}
	Let $A$ and $B$ be two convex sets in $\RRP^d$.  If $A$ and $B$ are
	consistent, then $A \cap B$ consists of at most one connected component.  If
	$A$ and $B$ are inconsistent, then $A \cap B$ consists of exactly two
	connected components.
\end{lemma}

\begin{proof}
	If $A$ and $B$ are consistent, we can regard them as convex sets in Euclidean
	space.  Hence they are either disjoint, or their intersection is convex,
	hence connected.

	Conversely, if $A \cap B$ consists of at most one connected component, they
	can be lifted to two convex sets $A'$ and $B'$ of $\mathbb{S}^d$.  We may
	assume that $A' \cap B' = \emptyset$.  If this is not the case, it suffices
	to replace $B'$ by its antipodal image.  By the spherical hyperplane
	separation theorem, $A'$ and $B'$ is separated by a spherical hyperplane.  We
	then project back to $\RRP^d$, taking this separating hyperplane at infinity.
	This gives us an affine chart that contains both $A$ and $B$.

	Consequently, if $A$ and $B$ are inconsistent, $A \cap B$ will have at least
	two connected components.  To prove that there are exactly two components, we
	only need to work in dimension $d=2$.  Higher dimensional cases follow by
	restricting to a $2$-dimensional subspace.

	\begin{figure}[hbt]
  	\centering
  	\includegraphics[width=.3\textwidth]{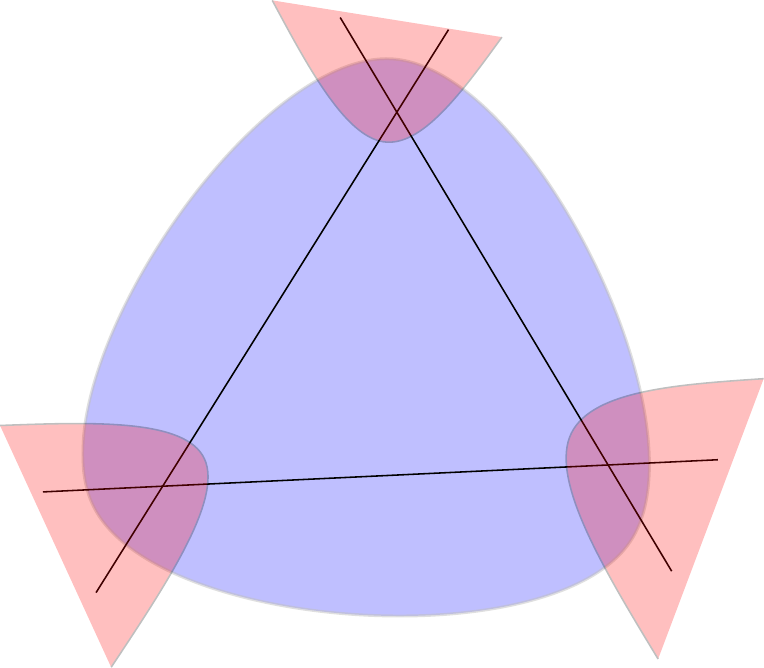}
  	\caption{
  		\label{fig:3components}
		}
	\end{figure}

	For the sake of contradiction, assume that $A \cap B$ has three components,
	and take a point from each of them.  The three points determine six segments.
	In an affine chart that contains $A$, the three bounded segments are
	contained in $A$, the three unbounded segments are contained in $B$.  The
	situation is illustrated in Figure~\ref{fig:3components}, where $A$ is blue
	and $B$ is red.  Since $B$ is convex, there should be a projective line
	avoiding $B$.  To avoid finite points of $B$ in the chosen affine chart, such
	a line must be parallel to one of the three lines spanned by the points.  But
	this means that this line intersects $B$ at infinity.  Hence such a
	projective line does not exist, contradicting the convexity of $B$.
\end{proof}

If $A$ and $B$ are two inconsistent convex regions in $\RRP^2$, then $\partial
A \cap \partial B$ consists of at most four connected components, at most two
on the boundary of each connected component of $A \cap B$.  Otherwise, either
the interior or the closure of $A \cap B$ would consist of more than two
connected components, contradicting the lemma above.

A particular case is when $A = \HHP^2$ and $B$ is a polygon, and their boundary
intersect at the vertices of $B$, i.e.\ $B$ is weakly ideal to $\HSP^2$.  In
this case, Lemma~\ref{lem:convcomp} implies

\begin{corollary}\leavevmode
	\begin{itemize}
		\item Any polygon strongly ideal to $\HSP^2$ with at least three vertices
			is disjoint from $\dSP^2$.

		\item Any polygon ideal to $\HSP^2$ with at least five vertices is
			strongly ideal in $\HSP^2$.

		\item A weakly ideal polygon $P$ has three or four vertices, at most two in
			each connected component of $\partial P \cap \HHP^2$. 
	\end{itemize}
\end{corollary}

A dual version of this corollary was proved in~\cite{chen2017}.  Notice that
there is only one possibility for a weakly ideal triangle; see
Figure~\ref{fig:weaktriangle}.

\begin{figure}[hbt]
  \centering
  \includegraphics[width=.3\textwidth]{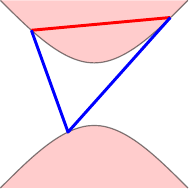}
  \caption{
  	\label{fig:weaktriangle}
	}
\end{figure}

Moreover, it is known that every connected component of $A \cap B$ is
convex~\cite{toda2010}.

\begin{lemma}\label{lem:contractible}
	Let $A$ and $B$ be two inconsistent convex sets in $\RRP^d$, and $C_1$,
	$C_2$ be the connected components of $A \cap B$.  Then $\partial C_i \cap
	\partial A$ and $\partial C_i \cap \partial B$, $i=1,2$, are all
	contractible.
\end{lemma}

\begin{proof}
	We only need to argue for $\partial C_2 \cap \partial B$.  The other cases
	follow similarly.
	
	We work in an affine chart containing $A$; thus it does not contain $B$.  Let
	$p \in \partial C_1 \cap \partial B$.  Then for any $q \in \partial C_2 \cap
	\partial B$, the bounded closed segment $[pq]$ is disjoint from the interior
	of $B$, hence also from the interior of $C_2$.  On the other hand, any $q'
	\in \partial C_2 \setminus \partial B$ (if not empty) is in the interior of
	$B$, hence $[pq']$ must intersect the interior of $C_2$.  In other words,
	$\partial C_2 \cap \partial B$ is the part of $\partial C_2$ ``visible'' from
	$p$, which must be contractible as $C_2$ is.

	The proof for $d=2$ is illustrated in Figure~\ref{fig:contractible}.
\end{proof}

\begin{figure}[htb]
  \centering
  \includegraphics[width=.3\textwidth]{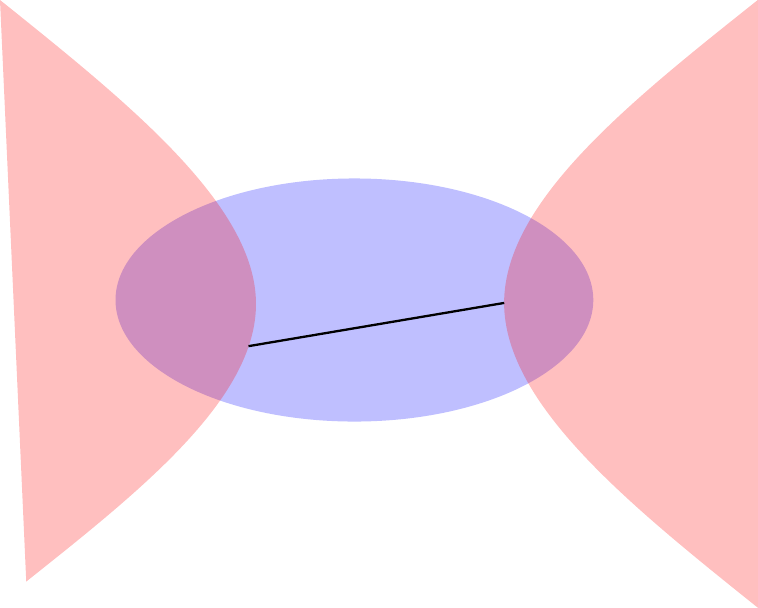}
  \caption{
  	\label{fig:contractible}
	}
\end{figure}

We call an (open) face (that is vertex, edge or facet) $F$ of $P$
\emph{interior} if $F \subset \overline{\HHP^3}$, or \emph{exterior} if $F
\subset \dS^3$.  For example, every vertex of an ideal polyhedron is interior,
and every face of a strongly ideal polyhedron is interior.  An edge of an ideal
polyhedron is either interior or exterior.  Let $\mathcal{I}(P)$ be the union
of interior faces, and $\mathcal{E}(P)$ be the union of exterior faces.

\begin{proposition}[Condition~\ref{con:2cycles} and more]\label{prop:Hstruct}
	Let $P$ be a polyhedron weakly ideal to $\HHP^3$.  Then $\mathcal{I}(P)$
	consists of two connected components, both contractible.  A component is
	homeomorphic to a closed disk if it contains at least three vertices.
	Vertices in each component induce an outerplanar graph.  Moreover,
	$\mathcal{E}(P)$ consists of disjoint open segments; there is no exterior
	facet.
\end{proposition}

\begin{proof}
	Assume a triangular facet $F$ that is not interior, then it would be a weakly
	ideal in $\Span(F)$.  Hence $P$ has no exterior facet.  The only exterior
	faces are edges.  Then we observe from Figure~\ref{fig:weaktriangle} that,
	whenever $F$ is not interior, $F \cap \partial \HHP^3$ is an arc which is
	homeomorphic to the unique interior edge of $F$.  This homeomorphism induces
	an homotopy from $\partial P \cap \overline{\HHP^3}$ to $\mathcal{I}(P)$.  The
	former is contractible by Lemma~\ref{lem:contractible}, hence so is
	$\mathcal{I}(P)$.

	The vertices of $P$ are all on the boundaries of $\partial P \cap
	\overline{\HHP^3}$ (and of $\mathcal{I}(P)$).  Otherwise, through a vertex in
	the interior of $\partial P \cap \overline{\HHP^3}$, we can find an hyperplane
	$H$ such that the intersection of $H \cap \partial P$ and $H \cap \partial
	\HHP^3$ consists of more than two connected components, contracting
	Lemma~\ref{lem:convcomp}.  Hence the vertices of each component induce an
	outerplanar graph.  If there are at least three vertices in this component,
	the boundary edges form a Hamiltonian cycle (of the induced graph), so the
	outerplanar graph is $2$-connected.  We then conclude that the component of
	$\mathcal{I}(P)$ is homeomorphic to a disk.
\end{proof}

This proves the necessity of Condition~\ref{con:2cycles} since the boundary
edges of the $2$-connected outerplanar graphs form a cycle cover consisting of
two cycles.

Induction in higher dimensions yields that the union of interior faces consists
of two contractible components, and vertices and edges in each component form a
2-connected graph.  However, it is possible that a component is not
homeomorphic to the $(d-1)$-ball, as shown in the following example.

\begin{example}
	Consider the eight points $[\pm \sqrt{2}, \pm 1, 0, 0]$,
	$[2(1-\epsilon+\epsilon^2), 0, \pm(1-\epsilon), \pm \epsilon]$, taking all
	possible combinations of $\pm$.  Their convex hull is a $4$-dimensional
	polytope weakly inscribed to the two-sheeted hyperboloid defined by the
	equation $-x_0^2+x_1^2+x_2^2+x_3^2=-1$.  If $\epsilon$ is sufficiently small,
	the polytope has two interior facets in $\HHP^4_+$, whose intersection is a
	single edge connecting $[\sqrt{2}, \pm 1, 0, 0]$.  Hence the union of
	interior faces is not homeomorphic to a $3$-ball.
\end{example}

\subsection{Angle conditions} \label{sec:extrinsic}

We now prove that the conditions in Theorem \ref{thm:angle} are necessary.

\begin{proposition}
	\[
		\Theta(\cP_{p,q}) \subseteq \cA_{p,q}.
	\]
\end{proposition}

For that, we need to verify all the conditions defining an admissible graph.

We color interior edges by red, and exterior edges by blue, and prove the angle
conditions with respect to this coloring.  In the previous part we have seen
that this coloring coincides with the combinatorial description in
Theorem~\ref{thm:angle}.

Among the conditions that involve angles, Condition~\ref{con:range} comes from
the definition of angle.  To see Condition~\ref{con:vertex} we need the vertex
figures.

\medskip

Recall that a \emph{horosphere} in $\HHP^d$ based at an ideal point $x \in
\partial\HHP^d$ is a hypersurface that intersects orthogonally all the geodesics
emerging from $x$.  In the projective model of $\HHP^3$, horospheres appear as
flattened spheres tangent to the hyperbolic boundary.  Similarly, a horosphere
in $\dSP^d$ based at $x \in \partial\dSP^d$ is a hypersurface that intersects
orthogonally all the geodesics emerging from $x$.  Horospheres in $\HHP^d$ and
$\dSP^d$ are paired through polarity: the set in $\dSP^d$ polar to a horosphere
in $\HHP^d$ is a horosphere with the same base point, and vice versa.  See
Figure~\ref{fig:horocycles}.  In the following, by a horosphere in $\HSP^d$, we
mean a horosphere in $\HHP^d$ or in $\dSP^d$.

\begin{figure}[hbt] 
  \centering 
  \includegraphics[width=.4\textwidth]{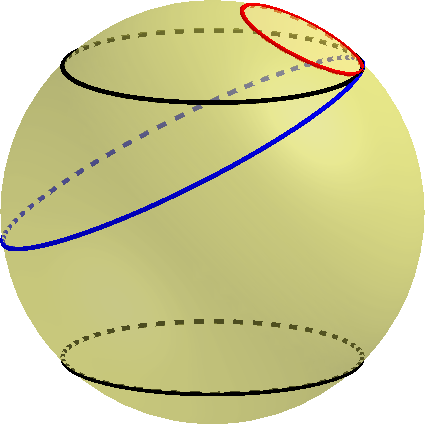}
 	\caption{
 		Horocycles in $\HH^2$ (red) and in $\dS^2$ (blue).  The black circles are
 		the boundaries of $\HSP^2$.  \label{fig:horocycles}
	}
\end{figure}

Now consider an ideal polyhedron $P$.  The \emph{vertex figure} of $P$ at a
vertex $v$, denoted by $P/v$, is the projection of $P$ with respect to $v$.
$P/v$ is therefore a polygon in $\RRP^2$.  A chart is provided by a horosphere
in $\HSP^3$ based at $v$.  This chart sends $v^\perp$ to the line at infinity,
hence does not contain $P/v$ unless the neighborhood of $v$ in $\partial P$ is
contained in $\HHP^3$ or in $\dSP^3$.

\begin{figure}[hbt]
  \centering
  \includegraphics[width=.3\textwidth]{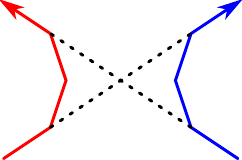}
  \caption{
  	A typical vertex figure. \label{fig:vertexfig}
	}
\end{figure}

Figure~\ref{fig:vertexfig} shows a typical situation of $P/v$ not contained in
the horosphere.  The vertices of the polygon $P/v$ correspond to the edges of
$P$ adjacent to $v$.  If one walks along the polygon, it can be arranged (as in
Figure~\ref{fig:vertexfig}) that he turns anti-clockwise at the vertices
corresponding to red edges, and clockwise at the vertices corresponding to blue
edges.  Indeed, the turning direction switches when the walker passes through
infinity.  The turning angles are then the dihedral angles $\theta$ at the
corresponding edges, taking anti-clockwise turns as positive, and clockwise
turns as negative.  Condition~\ref{con:vertex} then follows immediately.

\medskip

We now verify the last angle condition.

\begin{proposition}[Condition~\ref{con:negsum}]
  If $P$ is weakly ideal, then $\Theta(P)$ sum up to at least $-2\pi$ over the
  blue edges, and $-2\pi$ is achieved only when $p<2<q$.
\end{proposition}

\begin{proof}
  It is clear from the vertex figure that the sum over blue edges is $-2\pi$
  when $p=1$.  Hence we focus on the case $2 \le p \le q$ and prove that the
  sum over blue edges is strictly larger than $-2\pi$.

  If $2\le p\le q$, the polar $P^*$ of $P$ is a compact polyhedron in $\dSP^3$.
  All its faces are light-like (isotropic).  The edges of $P^*$ polar to the
  blue edges of $P$ form a closed space-like polygonal curve $\gamma$, whose
  vertices are polar to the non-interior facets of $P$.

  The polygonal curve $\gamma$ is a ${\mathcal T}$-geodesic for the induced
  HS-structure on $P^*$; see \cite[Definition 3.4]{schlenker2001}. It then
  follows from point C. in \cite[Theorem 1.5]{schlenker2001} that $\gamma$ has
  length strictly less than $2\pi$.  The proposition follows by polarity.
\end{proof}

\subsection{Metric conditions}\label{sec:intrinsic}

We show in this part that the conditions of Theorem \ref{thm:induced} are
necessary.

\begin{theorem}
	\[
		\Delta(\cP_{p,q}) \subseteq \cM_{p,q}.
	\]
\end{theorem}

If $P$ is $(p,q)$-ideal, we have argued that $\Delta(P)$ consists of one or two
copies of $\HHP^2$ and a de Sitter surface.  It remains to verify that the
pieces are glued along their ideal boundaries by CPP ($C^1$ piecewise
projective) maps with positive break points at the vertices of $P$.  We may
focus on the boundary component of $\dS^2_\ell$, $\ell \ge 0$, consisting of
$q>1$ vertices of $P$.  Let $\gamma$ be the map that sends $\partial \HHP^2$ to
this boundary.

We can identify $\partial\HHP^2$ to the real projective line $\RRP^1$.  Let
$v_0 < v_1 < \dots < v_p = v_0$, in this order, be the $p$ vertices of $P$.
They divide $\RRP^1$ into $p$ segments $[v_i, v_{i+1}]$.  Each segment
corresponds to a segment of $\partial \HSP^2$ in the interior of a face
triangle weakly ideal to $\HSP^2$.  Hence for each $i$, $0 \le i < p$, the
restriction $\gamma_i = \gamma|_{[v_i, v_{i+1}]}$ coincides with the
restriction of an element of $\PSL(2,\RR)$.  This proves that
$\gamma$ is piecewise projective.

To study differentiabilities, we will follow the work of
Martin~\cite{martin2005} on CPP homeomorphisms of $\RRP^1$.  Note that
$\partial \dS^2_\ell$ is not projectively equivalent to $\RRP^1$.  But our CPP
maps are local homeomorphisms.  Hence the definitions and many results
from~\cite{martin2005} remain valid for our case.

\begin{definition}[\cite{martin2005}]
	For $x \in \RRP^1$, let $\gamma^\leftarrow_x \in \PSL(2,\RR)$
	(resp.\ $\gamma^\rightarrow_x \in \PSL(2,\RR)$) be the left (resp.\
	right) germ of a piece-wise projective map $\gamma$.  The projective
	transformation $D_x\gamma = (\gamma^\rightarrow_x)^{-1} \circ
	\gamma^\leftarrow_x$ is called the \emph{shift} of $\gamma$ at $x$.
\end{definition}


The projective transformations $\gamma^\leftarrow_x$, $\gamma^\rightarrow_x$
and $D_x\gamma$ extend uniquely to isometries of $\HSP^2$.  We abuse the same
notations for these extensions.  The shift $D_x\gamma$ measures how much
$\gamma$ fails to be projective at $x$.  It can be understood as the holonomy
of the HS structure along a curve going around $x$.  The following lemma
reveals the relation between the shift and the differentiability of $\gamma$.

\begin{lemma} \label{lem:C1}
	$\gamma$ is $C^1$ at $x$ if and only if $D_x\gamma$ preserves the horocycles
	based at $x$. 
\end{lemma}

\begin{proof}
	The proof of~\cite[Proposition 2.3]{martin2005} can be used here, word by
	word, to prove the ``only if''.  For the ``if'' part, assume $x=0$.  Note
	that $u:=D_0\gamma: \RRP^1 \to \RRP^1$ fixes $0$.  If (the extension of)
	$u$ preserves the horocycles based at $0$, it must be of the form $u(t) =
	t/(ct+1)$.  Hence $u'(0)=1$, therefore $\gamma$ is $C^1$.
\end{proof}

Then $\gamma$ being $C^1$ follows from the following proposition.

\begin{proposition}\label{prop:C1}
	For any $x \in \partial \HHP^2$, $D_x\gamma$ preserves the horocycles based at
	$x$.
\end{proposition}

\begin{proof}
  As a measure of how much $\gamma$ fails to be projective at $x$, $D_x\gamma$
  must be trivial except at vertices of of $P$.  Let $v \in V^+$ be a vertex of
  $P$.

  We first note that in a neighborhood of $v$, exactly two faces of $\partial
  P$, say $F^\leftarrow$ and $F^\rightarrow$, have non-empty intersections with
  $\partial \HHP^3$.  This can be seen from the vertex figure $P/v$.  Recall
  that vertices (resp.\ edges) of $P/v$ correspond to edges (resp.\ faces) of
  $P$ adjacent to $v$.  In the affine chart of $\RRP^2$ provided by a
  horosphere at $v$, the tangent plane of $\partial \HHP^3$ at $v$ is sent to
  infinity, and intersects $P/v$ in exactly two edges, corresponding to
  $F^\leftarrow$ and $F^\rightarrow$.  In particular, no vertex of $P/v$ is at
  infinity; otherwise such a vertex would correspond to an edge $e$ of $P$
  tangent to the quadratic $\partial \HHP^3$, and the other end of $e$ can not
  lie on the same quadratic.
  	
  Our gluing map $\gamma$ coincides with $\gamma^\leftarrow_v$ on $\partial
  \HSP^3 \cap F^\leftarrow$, and with $\gamma^\rightarrow_v$ on $\partial
  \HSP^3 \cap F^\rightarrow$.  	
  
  $F^\leftarrow$ and $F^\rightarrow$ extend to half-planes bounding a dihedral
  angle $\Phi$ containing $P$.  The boundary of $\Phi$ is isometric to
  $\HSP^2$.  Let $\tilde h$ be a horocycle in $\HSP^3$ based at $v$.  Then the
  intersection $h = \partial \Phi \cap \tilde h$ give a horocycle in $\HSP^2$.
  On the other hand, $h' = \partial P \cap \tilde h$ give a horocycle in
  $\partial P$ based at $v$.  In a neighborhood of $v$, $h'$ lies within
  $F^\leftarrow$ and $F^\rightarrow$.
  
  We then conclude that, for any horocycle $h$ in $\HSP^2$, we have $h' =
  \gamma^\leftarrow_x(h) = \gamma^\rightarrow_x(h)$, i.e.\ $D_x\gamma(h)=h$.
\end{proof}

Proposition~\ref{prop:C1} can be interpreted as horocycles ``closing up'' after
going around a vertex.  In Rivin's characterization of polyhedra strongly
inscribed in the sphere, the same phenomenon is reflected by the shearing
coordinates summing up to $0$ around each vertex.  We can do the same with a
proper definition of shearing coordinates.

Note that three vertices on $\partial \HSP^2$ determine a unique strongly ideal
(hyperbolic) triangle.  Hence for a given ideal HS structure $\delta \in
\cM_{p,q}$, we can replace every triangle in $\delta$, if not already strongly
ideal, by the unique hyperbolic triangle with the same vertices.  The result is
a hyperbolic structure $\eta$ (a triangulation with hyperbolic simplices) of
the $n$-times punctured $\SS^2$ (not embedded).  We define the \emph{HS
shearing} (or simply shearing) along an edge of $\delta$ as the hyperbolic
shearing (see \cite{penner:decorated1, penner:decorated2}) along the
corresponding edge of $\eta$.

Shearing can be easily read from the vertex figures.  First note that the
vertex figure of $\eta$ at a vertex $v$ can be obtained from that of $\delta$
by replacing the segments through infinity by the unique other segments with
the same vertices.  For example, Figure~\ref{fig:hypvertexfig} is obtained from
Figure~\ref{fig:vertexfig}.  Let $e$ be an edge of $P$ adjacent to a vertex
$v$.  The shearing along $e$ then equals to the logarithm of the length ratio
of the segments adjacent to $e$ in the vertex figure of $\eta$ at $v$.

\begin{figure}[hbt]
  \centering
  \includegraphics[width=.3\textwidth]{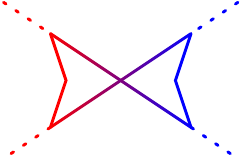}
  \caption{
  	The vertex figure of the hyperbolic structure corresponding to the vertex
  	figure of the HS structure shown in Figure~\ref{fig:vertexfig}.
  	\label{fig:hypvertexfig}
	}
\end{figure}

Horocycles in $\eta$ close up if and only if the hyperbolic shearings sum up
to $0$ over the edges adjacent to $v$.  Then we see from the vertex figure that
the horocycles in $\delta$ also ``close up''.  And by definition, the HS
shearings of $\eta$ must also sum up to $0$.  Different triangulation of
$\delta$ would yield a different hyperbolic metric $\eta$.  But for a fixed
triangulation, it is well-known (see \cite{penner:decorated1,
penner:decorated2}) that the hyperbolic shearing on the edges of $\eta$ provide
a coordinate system for the hyperbolic structure.  Hence the HS shearing on the
edges of $\delta$ provide a coordinate system for the ideal HS metrics.

\medskip

Now back to the proof of necessity.  Proposition~\ref{prop:C1} asserts that the
$D_x\gamma$ are parabolic transformations for every $x \in \RRP^1$.  Consider
the projective transformation $u_x: t \mapsto 1/(t-x)$ sending $x$ to infinity.
Then the conjugate $u_x D_x(\gamma) u_x^{-1}$ is a translation of the form $t
\mapsto t+d_x(\gamma)$.  We have $d_x(\gamma)=0$ at projective points.  At break points:

\begin{lemma} \label{lem:C2}
	$d_x(\gamma)<0$ (resp.\ $>0$) if $x$ is a positive (resp.\ negative) break points.
\end{lemma}

\begin{proof}
	We may assume $x=0$, then $u_0=1/t$.  We already figured that $D_0(\gamma)$
	is of the form $D_0(\gamma): t \mapsto t/(ct+1)$.  So the conjugate $u_0
	D_0(\gamma) u_0^{-1}$ has the form $t \mapsto t+c$, i.e.\ $d_0(\gamma)=c$.
	On the other hand, an elementary computation shows that the second derivative
	$\frac{d^2}{d t^2}|_{t=0} D_0(\gamma)=-2c$.
\end{proof}

In the half-space model of $\HHP^2$, let $h$ be a horocycle based at $x$.  Then
$d_x(\gamma)<0$ (resp.\ $>0$) if and only if $D_x\gamma$ moves points on $h$ in
the clockwise (resp.\ anti-clockwise) direction.  We are now ready to prove
that every break point of $\gamma$ is positive.

\begin{proposition}\label{prop:C2}
	$d_v(\gamma) < 0$ at every vertex $v$ of $P$.
\end{proposition}

\begin{proof}
	We keep the definitions and notations in the proof of
	Proposition~\ref{prop:C1}.  We can recover $P$ by truncating the dihedral
	angle $\Phi$ with planes through $v$.  From the vertex figure, we observe the
	effect of a truncation on a horocycle $h$ based at $v$:  it replaces a
	segment of $h$ with a shorter one. See Figure~\ref{fig:truncate}.  Hence
	$D_v\gamma = (\gamma^\rightarrow_v)^{-1} \circ \gamma^\leftarrow_v$ moves
	points on $h$ in the clockwise direction, i.e.\ $d_v(\gamma)<0$.
\end{proof}

\begin{figure}[hbt] 
  \centering 
  \includegraphics[width=.3\textwidth]{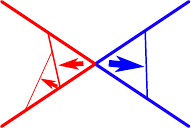}
 	\caption{
 		Each truncation replaces a horocyclic segment with a shorter
 		one.  \label{fig:truncate}
	}
\end{figure}

\begin{remark}\label{rem:martin}
  Lemma~6.5 of~\cite{martin2005} asserts that $d_x(\gamma)$ equals the change
  of length of a well chosen segment of horocycle based at $x$.  Up to a
  scaling, this also suffices for us to conclude that $d_v(\gamma)<0$.
\end{remark}



\section{Properness}\label{sec:properness}
The following theorem states that the maps $\Theta$ and $\Delta$ are proper.

\begin{theorem}
	Consider a sequence of polyhedra $(P_k)_{k \in \NN}$ diverging in
	$\cP_{p,q}$, then $\theta_k=\Theta(P_k)$ diverges in $\cA_{p,q}$ and
	$\delta_k = \Delta(P_k)$ diverges in $\cM_{p,q}$.
\end{theorem}

Up to hyperbolic isometries, we may fix three vertices for every polyhedron in
$(P_k)$.  As $\partial \HHP^3$ is compact, we may assume that vertices of $P_k$
have well defined limits by taking a subsequence.  But the limit of $P_k$,
denoted by $P_\infty$, is not a $(p,q)$-ideal polyhedron, since the sequence is
diverging.

Hence in the limit, $P_\infty$ must fail to be \emph{strictly} convex at some
vertex $v$.  Let $P'_\infty$ be the convex hull of all the other vertices.
There are three possibilities, namely that $v$ is in the relative interior of a
vertex, an edge or a facet of $P'_\infty$.  But every straight line intersect a
quadratic surface in at most two points, hence an ideal vertex can not be in
the relative interior of an edge.  Thus we only need to consider the remaining
two possibilities.

\begin{remark}
	For strongly ideal polyhedra, an ideal vertex can not lie in the interior of
	a facet.  Hence there is only one possibility to consider.
	See~\cite{danciger2014}.
\end{remark}

\begin{proposition}
	If some vertices of $(P_k)_{k \in \NN}$ converge to the same vertex of
	$P_\infty$, then the limit graph $\theta_\infty$ violates
	Condition~\ref{con:vertex} or~\ref{con:negsum}, and some break points
	in~$\delta_k$ merge into a single break point in the limit metric
	$\delta_\infty$.
\end{proposition}

\begin{proof}
	The divergence of the induced metrics follows immediately from the
	correspondence between vertices of $P$ and break points in $\Delta(P)$.
	Hence we will focus on the divergence of the admissible graphs.

	Note that the ideal boundary of $\HHP^3$ can be seen as two copies of $\HHP^2$
	identified along their ideal boundaries.  With our choice of affine chart for
	the projective model, the two copies of $\HHP^2$ appear as the two sheets of
	the hyperboloid $x_0^2+x_1^2-x_2^2 = -1$.

	The vertex set of each $P_k \in \cP_{p,q}$ then corresponds to two point sets
	$V_k^+$ and $V_k^-$ in $\HHP^2$ of cardinality $p$ and $q$ respectively.
	They converge to two point sets $V^+_\infty$ and $V^-_\infty$ of cardinality
	$p'$ and $q'$ respectively, corresponding to the vertices of $P_\infty$.  If
	some vertices of $(P_k)$ converges to the same vertex of $P_\infty$, we must
	have $p'+q' < p+q$.  Up to hyperbolic isometries, we may assume three fixed
	vertices shared by all $P_k$, hence $3 \le p'+q'$.

	Consequently, the graph $\Theta(P_\infty)$ has at least three vertices, but
	stricly less vertices than $\theta_\infty$.  In fact, it is obtained by
	contracting vertices of $\theta_\infty$.  Recall that $\theta_k=\Theta(P_k)$
	and their limit $\theta_\infty$ are vectors of dimension $\binom{|V|}{2} =
	\binom{p+q}{2}$.

	Assume that a set of vertices $S \subset V$ is merged into a vertex of
	$P_\infty$.  If $S \ne V^+$ or $V^-$, we have on the one hand
	\[
		\sum_{u \notin S, v \in S} \theta_\infty(u,v) = 0.
	\]
	On the other hand, as the limit of $\theta_k$, Condition~\ref{con:vertex}
	asserts that
	\[
		\sum_{u \in V} \theta_\infty(u,v) = 0
	\]
	for any $v \in S$.  Comparing the two sums, we conclude that
	\[
		\sum_{u,v \in S} \theta_\infty(u,v) = 0.
	\]
	Now assume that $\theta_\infty$ is $(p,q)$-admissible.  Since $S$ are
	vertices on the same polar circle, $\theta_\infty(u,v)$ is non-negative,
	hence must be $0$, if $u,v \in S$.  In other words, $S$ induce an empty
	graph, contradicting the fact that vertices in $S$ are consecutive in a
	cycle.

	If $S = V^+$ or $V^-$, we must have $p>1$.  But it is easy to conclude that
	the sum over negative weights in $\theta_\infty$ is $-2\pi$, contradicting
	Condition~\ref{con:negsum}.
\end{proof}

\begin{proposition}
	If a vertex $v$ of $P_k$ converges to a vertex of $P_\infty$ that is contained
	in a unique supporting plane, then $v$ is an isolated vertex in the limit
	graph $\theta_\infty$, and $v$ is not a break point in the HS structure
	induced by $\delta_\infty$.
\end{proposition}

\begin{proof}
	Under the assumption of the proposition, every face of $P_\infty$ adjacent to
	$v$ must lie in this unique supporting plane.  Otherwise, the supporting
	plane of the face would be another supporting plane containing $v$.
	Then the dihedral angles vanish on all the edges incident to $v$.  Since $v$
	is in the interior of a weakly ideal HS triangle, the gluing map is
	projective at $v$.
\end{proof}

A special case is of particular importance for us:  If $\partial P_k$ converges
to a double cover of a plane, then the limit polyhedron $P_\infty$ is equal to
$\RRP^3$.  In this case, every vertex is ``flat'': $\theta_\infty$ is
identically $0$ (empty graph), and $\delta_\infty$ is the double cover of
$\HSP^2$.  We call this polyhedron a \emph{flat polyhedron}.

\begin{proposition}
	If two faces of $P_\infty$ intersect in their relative interiors and span a
	plane, then $P_\infty$ is flat.
\end{proposition}

\begin{proof}
	Under the assumption of the proposition, the only plane that avoids the
	interior of $P_\infty$ is the plane spanned by the two intersecting faces.
	Hence every face admits this plane as the supporting plane.  In other words,
	every face lies in this plane.
\end{proof}

\section{Rigidity}\label{sec:rigidity}
\subsection{The infinitesimal Pogorelov map}

We recall here the definition of the infinitesimal Pogorelov map, as well as
its key properties. We refer to \cite{schlenker1998} for the proofs, see in
particular D\'efinition 5.6 and Proposition 5.7 in \cite{schlenker1998}. Other
relevant references are \cite{fillastre2,izmestiev:projective,iie,rcnp}. 

With affine charts containing weakly ideal polytopes, the hyperplane at
infinity $H_\infty$ is space-like.  Apart from the HS metric and the usual
Euclidean metric, the affine charts can also carry the Minkowski metric.  Then
the point $x_0=H_\infty^\perp$ is the ``center'' of the Minkowski space
$\RR^{2,1}$.  The set of light-like geodesics passing through $x_0$ is called
the \emph{light cone} at $x_0$, denoted by $C(x_0)$.

Let $U = \RRP^3 \setminus H_\infty$ be an affine chart, and $\iota: U \to
\RR^{2,1}$ be the projective embedding into the Minkowski space.  The
infinitesimal Pogorelov map $\Upsilon$ is then defined as the bundle map
$\Upsilon: T U \to T \RR^{2,1}$ over the inclusion $\iota : U \hookrightarrow
\RR^{2,1}$ as follows: $\Upsilon$ agrees with $d \iota$ on $T_{x_0} U$. For any
$x \in U \setminus C(x_0)$, and any vector $v \in T_x U$, write $v = v_r +
v_{\perp}$, where $v_r$ is tangent to the radial geodesic passing through $x_0$
and $x$, and $v_{\perp}$ is orthogonal to this radial geodesic, and define
\[
	\Upsilon(v) = \sqrt{\frac{\| \hat x \|_\HSP^2}{\| d\iota(\hat x) \|_{2,1}^2}}
	d\iota(v_r) + d\iota( v_{\perp}),
\]
where the norm $\| \cdot \|_\HSP$ in the numerator of the first term is the HS
metric, the norm $\| \cdot \|_{2,1}$ in the denominator is the Minkowski metric
and $\hat x$ is the normalized radial vector (so $\|\hat x\|^2_{2,1} = \pm 1$).

The key property of the infinitesimal Pogorelov map is the following (the proof
	is an easy computation in coordinates, that can be adapted from \cite[Lemma
3.4]{fillastre2011}).

\begin{lemma} \label{lem:pogorelov}
	Let $Z$ be a vector field on $U \setminus C(x_0) \subset \HSP^3$. Then $Z$ is
	a Killing vector field if and only if $\Upsilon(Z)$ (wherever defined) is a
	Killing vector field for the Minkowski metric on $\RR^{2,1}$.
\end{lemma}

In fact, the lemma implies that the bundle map $\Upsilon$, which so far has
only been defined over $U \setminus C(x_0)$, has a continuous extension to all
of~$U$. The bundle map $\Upsilon$ is called an \emph{infinitesimal Pogorelov
map}.

Next, the bundle map $\Xi: T\RR^{2,1}\to T\RR^{3}$ over the identity, which
simply changes the sign of the $n$-th coordinate of a given tangent vector, has
the same property: it sends Killing vector fields in $\RR^{2,1}$ to Killing
vector fields for the Euclidean metric on $\RR^{3}$. Hence the map $\Pi = \Xi
\circ \Upsilon$ is a bundle map over the inclusion $U \hookrightarrow \RR^3$
with the following property:

\begin{lemma} \label{lem:pogorelov-Euclidean}
	Let $Z$ be a vector field on $U \subset\HSP^3$. Then $Z$ is a Killing vector
	field if and only if $\Pi(Z)$ is a Killing vector field for the Euclidean
	metric on $\RR^3$.
\end{lemma}

The bundle map $\Pi$ is also called an infinitesimal Pogorelov map, since it is
an infinitesimal version of a remarkable map introduced by Pogorelov \cite{Po}
to handle rigidity questions in spaces of constant curvature. 

\subsection{Rigidity with respect to HS structures} Here, an
\emph{infinitesimal deformation} of $P$ associates a vector tangent to
$\partial \HHP^3$ to each vertex of $P$; the infinitesimal deformation is
trivial if it is the restriction of a global Killing field of $\HSP^3$.

\begin{proposition} \label{prop:rigidmetric}
  Let $P \in \cP_{p,q}$ and $\dot P$ be an infinitesimal deformation of $P$
  within $\cP_{p,q}$. If $\dot P$ does not change the HS structure $\Delta(P)$
  at first order, then $\dot P$ is trivial.
\end{proposition}

\begin{proof}
  As always, we work in an affine chart containing $P$.  Suppose that $\dot P$
  is a non-trivial infinitesimal deformation of $P$ that does not change, at
  first order, the HS metric induced on $P$.  Then the induced HS metric on
  each facet is constant at first order.  Hence for each facet $F$, there is a
  Killing field $\kappa_F$ such that the restriction of $\kappa_F$ to the
  vertices of $F$ is equal to the restriction of $\dot P$, and for two facets
  $F$ and $G$, $\kappa_F$ and $\kappa_G$ agree on the common edge of $F$ and
  $G$.
  
  Lemma \ref{lem:pogorelov-Euclidean} shows that $\bar \kappa_F =
  \Pi(\kappa_F)$ is the restriction of a Killing field of $\RR^3$, while it is
  clear that if $F$ and $G$ share an edge, then $\bar \kappa_F$ and $\bar
  \kappa_G$ agree on this edge. Therefore the restriction of $\bar \kappa_F$ to
  the vertices of $P$ define an isometric first-order Euclidean deformation of
  $P$.
  
  However, Alexandrov \cite{alexandrov1950} proved that convex polyhedra in
  $\RR^3$ are infinitesimally rigid: any first-order Euclidean isometric
  deformation must be the restriction of a global Killing vector field of
  $\RR^3$. So $\bar \kappa_F$ must be the restriction of a global Killing
  vector field $\bar \kappa$.  Lemma \ref{lem:pogorelov-Euclidean} therefore
  implies that $\kappa_F$ are the restriction to the faces of $P$ of a global
  Killing vector field $\kappa = \Pi^{-1}(\bar \kappa)$, which contradicts our
  hypothesis.
\end{proof}

\subsection{Shape parameters associated to edges} We can identify $\partial
\HHP^3$ with the extended complex plane $\CC\mathrm{P}^1$, then vertices of an
ideal polytope $P$ can be described by complex numbers.  By subdividing
non-triangular facets if necessary, we may assume that facets of $P$ are all
triangles.  Let $z_1z_2z_3$ and $z_2z_1z_4$ be two facets of $P$ oriented with
outward pointing normal vectors.  The \emph{shape parameter} on the common
(oriented) edge $z_1z_2$ is the cross ratio
\[
	\tau=[z_1,z_2;z_3,z_4]=\frac{(z_1-z_3)(z_2-z_4)}{(z_2-z_3)(z_1-z_4)}.
\]

Recall that each triangular facet of $P$ determines a strongly ideal triangles
with the same vertices.  The two (oriented) hyperbolic triangles corresponding
to $z_1z_2z_3$ and $z_2z_1z_4$ form a hyperbolic dihedral angle at their common
edge $z_1z_2$.  Let $\phi$ denote the \emph{hyperbolic} exterior angle at
$z_1z_2$.  Then the shape parameter $\tau$ has a geometric interpretation: it
can be written in the form of $\tau=\exp(\sigma+i\phi)$, where $\sigma$ is the
shearing between the two hyperbolic triangles.

The angle $\phi$ can be read from the hyperbolic vertex figure (see
Section~\ref{sec:intrinsic}).  If one walks along the polygonal curve, in the
same direction as we specified for reading $\theta$ (see
Section~\ref{sec:overview}), then $\phi$ is nothing but the turning angle at
every vertex, taking anti-clockwise turns as positive, and clockwise turns as
negative.

Let $v$ be a vertex of $P$, and let $\tau_1, \tau_2, \cdots, \tau_k$ be the
shape parameters associated to the edges of $P$ adjacent to $v$, in this cyclic
order.  The following relations, which holds for strongly inscribed polyhedra,
also holds for the weakly inscribed $P$.
\begin{align}
	\prod_{i=1}^k \tau_i & =1,\label{eq:shape1}\\
	\sum_{j=1}^k \prod_{i=1}^j \tau_i & =0.\label{eq:shape2}
\end{align}
Both equations can be easily understood by considering the hyperbolic vertex
figure at $v$: \eqref{eq:shape1} follows from the fact that $\sum\sigma_i=0$
while $\sum\phi_i$ is a multiple of $2\pi$.  \eqref{eq:shape2} is just saying
that the vertex figure, considered as a polygonal curve in the Euclidean plane,
closes up.

The shape parameters determine the local geometry (angle and shearing) at every
edge, hence completely describe the polyhedron.  A small perturbation in the
shape parameter subject to \eqref{eq:shape1} and \eqref{eq:shape2} corresponds
to a deformation of $P$ into another weakly ideal polyhedron.  Indeed, the
convexity is stable under a small perturbation, then \eqref{eq:shape1} and
\eqref{eq:shape2} guarantee that the hyperbolic vertex figures are closed
polygonal curves, hence they are vertex figures of a weakly inscribed
polyhedron.

\subsection{Rigidity with respect to dihedral angles}


We now have the necessary tools to prove the infinitesimal rigidity of weakly
ideal polyhedra with respect to their dihedral angles. 

\begin{proposition} \label{prop:rigidangle}
  Let $P \in \cP_{p,q}$ and $\dot P$ be an infinitesimal deformation of $P$
  within $\cP_{p,q}$. If $\dot P$ does not change the dihedral angles
  $\Theta(P)$ at first order, then $\dot P$ is trivial.
\end{proposition}

\begin{proof}
  Let $\dot P \in \CC^{p+q}$ be an infinitesimal deformation of $P$,
  represented by the velocity of the vertices in $\CC\mathrm{P}^1$.  Let $\dot
  \tau=(\dot \tau_e)_{e \in E} \in \CC^{|E|}$ be the corresponding first-order
  variation of the shape parameters associated to the edges.  Suppose that
  $\dot P$ does not change the dihedral angles of $P$ (at first order).  This
  means that for all $e \in E$, $\dot \tau_e/\tau_e$ is real, because the
  argument of $\tau_e$ is equal to the dihedral angle at the corresponding
  edge.

  Now consider the first-order variation $i\dot \tau=(i\dot \tau_e)_{e \in E}$
  of the shape parameters. A crucial observation is that, since the conditions
  (1) and (2) above are polynomial, $i\dot \tau$ again corresponds to a
  first-order deformation of $P$, which we can call $i \dot P$.  Now for all $e
  \in E$ , $i \dot \tau_e/\tau_e$ is imaginary. This means that in the
  first-order deformation $i \dot P$, the shear along the edges remains fixed
  (at first order).  So $i \dot P$ does not change, at first order, the
  HS-structure induced on $P$.

  By Proposition \ref{prop:rigidmetric}, $i \dot P$ is trivial, and it follows
  that the infinitesimal deformation $\dot P$ is also trivial. The result
  follows.
\end{proof}

\section{Topology}\label{sec:topology}
\subsection{Ideal polyhedra}

In this section we will conclude that $\Theta$ and $\Delta$ are homeomorphisms.
The first step is to prove that the domain $\cP_{p,q}$ is connected.

We work in a projective chart inconsistent with $\HH^3$, in which $\partial
\HH^3$ is the quadric of equation $x_2 = f(x_0,x_1) = \pm\sqrt{x_0^2+x_1^2+1}$.
The following lemma allows us to place any weakly $(p,q)$-ideal polyhedron in a
convenient position.

\begin{lemma}
	For any $P \in \cP_{p,q}$, there is an isometry $T$ of $\HH^3$ such that
	$T(P)$ contains the origin and the points $(0,0,\pm 1)$.
\end{lemma}

\begin{proof}
	The proof uses the Hyperplane Separation Theorem for hyperbolic space.  We
	sketch here a quite standard proof in the spirit of~\cite{boyd2004}, as some
	ingredient in this proof would be useful for us.
	
	Note that $\HH^3 \cap P$ consists of two disjoint components, denoted by
	$P^+$ and $P^-$, both are convex subsets of $\HH^3$.  Let $u \in P^+$ and $v
	\in P^-$ such that the hyperbolic distance between $u$ and $v$ achieves the
	minimum hyperbolic distance between $P^+$ and $P^-$.  This distance is
	necessarily finite, hence $u$ and $v$ are necessarily on the boundary
	$\partial P$.  We claim that the hyperbolic plane $H$ that perpendicularly
	bisects the segment $\overline{uv}$ separates $P^+$ and $P^-$, i.e.\ $P^\pm$
	are on different sides of $H$.  To see this, assume $u' \in P^+$ is on the
	same side of $H$ as $P^-$.  Then a perturbation of $u$ towards $u'$ would be
	closer to $v$, contradicting our choice of $u$ and $v$.

	Now let $T$ be the isometry of $\HH^3$ that sends the separating plane $H$ to
	infinity.  Then the polar point of $H$, which is contained in $P$, is sent to
	the origin in the interior of $T(P)$.  Moreover, the line through $u$ and $v$
	is sent to the $x_2$-axis.  In particular, the points $(0,0,\pm 1)$ are also
	in the interior of $T(P)$.
\end{proof}

\begin{proposition}
  $\cP_{p,q}$ is connected.
\end{proposition}

\begin{proof}
	Let $P \in \cP_{p,q}$.  Thanks to the previous lemma, we may assume that $P$
	contains the origin and the points $(0,0,\pm 1)$.  We now define a
	deformation of $P \in \cP_{p,q}$.
	
	If $v \in V_\pm$, define $v_t$, $t>1$, as follows:  $v_t=v$ if the
	$x_2$-coordinate of $v$ is smaller than $t$; otherwise, $v_t$ is a point of
	height $x_2=\pm t$ obtained by moving $v$ along the gradient of $f$ (for
	metric induced on the quadric by the Euclidean metric $dx_0^2+dx_1^2+dx_2^2$
in a chart) towards $(0,0,\pm 1)$.  We claim that the point set $V_t=\{v_t \mid
v \in V_+ \cup V_-\}$ remains in convex position for all $t > 1$.  If $v_t \in
V_t$ is at height $x_2=\pm t$, $v_t$ is on the circle $x_0^2+x_1^2=t^2-1$; the
convexity at $v_t$ is then immediate.  Otherwise, $v_t$ coincides with a vertex
$v$ of $P$, and we claim that a supporting hyperplane $H_v$ of $P$ at $v$ is
also a supporting hyperplane of $P_t$.  This can be seen by noting that, for
any $u \in \partial\HH_\pm^3$ on the same side of $H_v$ as $(0,0,\pm 1)$, as
long as $u$ is sufficiently close to $H_v$, the gradient of $f$ at $u$ points
away from $H_v$.  Because $(0,0,\pm 1) \in P$ and $H_v$ is supporting, no
vertex $v_t$ would move across $H_v$ as $t$ decreases.

	Define $P_t$ as the convex hull of $V_t$.  We see that $P_t = P$ for $t$
	sufficiently large.  As $t$ approaches $1$, the vertices of $P_t$ lie,
	eventually, on two horizontal planes $x_2=\pm t$.

	Now assume another polyhedra $P' \in \cP_{p,q}$, which also contains the
	origin and the points $(0,0,\pm 1)$.  For $t$ sufficiently close to $1$, both
	$P_t$ and $P'_t$ have vertices on the horizontal planes $x_2=\pm t$.
	Polyhedra with vertices on these two planes form a connected subset of
	$\cP_{p,q}$; indeed, any choice of $p$ and $q$ ideal points on these two
	planes determines uniquely such a polyhedra.  Hence we find a continuous path
	from between $P$ and $P'$, which proves the connectedness of $\cP_{p,q}$.
\end{proof}

Let $\cP_{p,q}^i$ denote the open subset of $\cP_{p,q}$ consisting of polyhedra
with an edge $1^+i^-$.  $\{\cP_{p,q}^i \mid 1 \le i \le q\}$ form an open cover
of $\cP_{p,q}$.  During the deformation $P_t$ in the previous proof, we may
rotate $V_-$ around the $x_2$-axis so that $1^+i^-$ remains an edge.  This
shows that

\begin{proposition}
	$\cP_{p,q}^i$ is connected.
\end{proposition}

\subsection{Admissible graphs}

Correspondingly, let $\cA_{p,q}^i$ denote the open subset of $\cA_{p,q}$
consisting of graphs $\theta$ with $1^+i^-$ as an edge (recall the vertex
labeling from Section~\ref{sec:define}).  That is, either $\theta(1^+i^-)<0$,
or $\theta(1^+i^-)=0$ but the graph remains polyhedral if we include $1^+i^-$
as an edge.  Then $\{ \cA_{p,q}^i \mid 1 \le i \le q \}$ form an open cover of
$\cA_{p,q}$.

\begin{proposition}
	For $1 \le i \le q$, $\cA_{p,q}^i$ is homeomorphic to $\RR^{2(p+q-3)}$.
\end{proposition}

We treat two cases separately.

\begin{proof}[Proof for the case $p<2<q$]
  We may take $i=1$ without loss of generality.  To construct a graph $\theta
  \in \cA_{p,q}^1$, we first assign positive weights then negative weights.

  For the positive weights, our task is to find an outerplanar graph $\theta^+$
  on the vertices of the $q$-cycle with only positive edge weights.  For the
  final result to be an element of $\cA_{p,q}^1$, we need the support of
  $\theta^+$ to contain the $q$-cycle, and the sum of the weights of $\theta^+$
  to be $\pi$.  This condition on the sum can be seen by adding up the weight
  sum of $\theta$ around all the vertices of the $q$-cycle, and noticing that
  the edges of $\theta^+$ are double counted.  Let $\cA^+$ be the set of all
  $\theta^+$ satisfying those conditions.

  Note that $\theta^+$ can be embedded in the plane in such a way that the
  $q$-cycle is embedded as a $q$-gon, and the other edges are embedded as
  non-crossing diagonals of the $q$-gon.  Let $\cA_0^+$ denote the set of
  positive graphs that are only supported on the $q$-cycle, and $\cA_1^+$ the
  set of positive graphs that are only supported on non-crossing diagonals.
  Then any $\theta^+ \in \cA^+$ can be written in the form $\theta^+ = (1-t)
  \theta^+_0 + t \theta^+_1$ where $\theta^+_0 \in \cA_0$, $\theta^+_1 \in
  \cA_1$, for some $0 \le t < 1$ (note the strict inequality!).

	It is easy to see that $\cA_0^+$ is a $(q-1)$-dimensional open simplex.
	Graphs in $\cA_1^+$ with the same combinatorics (that is, supported on the
	same diagonals) also form an open simplex, whose dimension is the cardinality
	of their support minus $1$.  $\cA_1^+$ is therefore a simplicial complex: The
	maximal simplices are of dimension $q-4$, corresponding to the maximal set of
	non-crossing diagonals, and they are glued along their faces corresponding to
	common subsets.  This simplicial complex is well-known as the boundary of a
	polytope, namely the \emph{associahedron}~\cite{lee1989}.

	Therefore, the closure of $\cA^+ = (1-t) \cA^+_0 + t \cA^+_1$, $0 \le t < 1$,
	is topologically the join of a $(q-1)$-ball and a $(q-4)$-sphere, hence
	homeomorphic to a $(2q-4)$-ball.  The openness of $\cA^+$ follows from the
	openness of $\cA^+_0$ and the strict inequality $t<1$.

	Once positive weights are assigned, the negative weights are uniquely
	determined since $p=1$ and all vertices of the $q$-cycle are connected to
	only one negatively weighted edge.  Hence $\cA_{p,q}^i$ is homeomorphic to
	$\RR^{2(p+q-3)}$.
\end{proof}

We need more ingredients for the case $2 \le p \le q$.

First, we claim that if a graph $\theta$ is admissible, and the negative
weights of $\theta$ sum up to $-2 \omega > -2 \pi$, then for any
$0<t<\pi/\omega$, the scaled graph $t\theta$ is also admissible.  The claim
follows from the following lemma, which guarantees that
Condition~\ref{con:range} is not violated after the scaling:

\begin{lemma}
  Let $-2\omega > -2\pi$ be the sum of negative weights of $\theta$.  Then any
  negative weight of $\theta$ is at least $-\omega$.
\end{lemma}

\begin{proof}
  We argue by contradiction and suppose that there is an edge $e$ with negative
  weight strictly less than $-\omega$. Let $v$ be an endpoint of $e$. The sum
  of the positive weights on the red edges $e_1,\cdots, e_k$ adjacent to $v$ is
  at least $\omega$. Let $v_1,\cdots, v_k$ be their endpoints different from
  $v$; $e$ is not adjacent to any of them.  Then the sum of the negative
  weights over the blue edges adjacent to $v_1,\cdots, v_k$ must be strictly
  less than $-\omega$.  We then conclude that the sum of negative weights is
  not $-2\omega$ as assumed, but strictly less.
\end{proof}

This lemma also proves the redundancy of Condition~\ref{con:negsum} and part of
Condition~\ref{con:range}.  Any weight function that satisfy
Conditions~\ref{con:sign} and~\ref{con:wsum} can be normalized to satisfy
Conditions~\ref{con:range}--\ref{con:negsum}.  Hence these conditions are
not present in the main Theorem~\ref{thm:weights}.

\begin{proof}[Proof for the case $2 \le p \le q$]
	We may take $i=1$ without loss of generality.  Fix a number $0<\omega<\pi$.
	We will prove that the set of admissible graphs in $\cA_{p,q}^1$ with
	negative weights summing up to $-2\omega$ is homeomorphic to
	$\RR^{2(p+q)-7}$.  To construct such a graph $\theta$, we follow the same
	strategy as before.  That is, first assign positive weights then negative
	weights.

	For the positive weights, we need to find $\theta^+$ that is the disjoint
	union of two outerplanar graphs, one on the vertices of the $p$-cycle, and
	the other on the vertices of the $q$-cycle, with only positive edge weights.
	Moreover, we need the sum of the weights of each disjoint component to be
	$\omega$.  Hence each component can be obtained by taking the $\theta^+$
	constructed for the case $p<2<q$, and multiplying its weights by a constant
	$\omega/\pi$.  The space of such $\theta^+$ is then homeomorphic to
	$\RR^{2(p+q-4)}$.

	We then propose an algorithm that assigns weights to negative edges and
	outputs an admissible graph in $\cA_{p,q}^1$.  This algorithm depends on one
	parameter taken in a segment, hence proves the proposition.

	Recall that vertices are labeled by $1^+, \dots, p^+$ and $1^-,\dots,q^-$
	respectively, following the boundary of the outerplanar subgraphs in a
	compatible direction.  Also recall that the weight $w_v$ of a vertex $v$ is
	the sum of weights $\theta(u,v)$ over all other vertices $u$.  The vertex
	weights change as we update the edge weights.  Before we proceed, the weights
	are positive on every vertex, because only positive weights are assigned.
	Our goal is to cancel them with negative weights.  We also keep track of two
	indices $i$ and $j$, initially both $1$.  At each step, we assign a negative
	weight to the edge $i^+j^-$.  Moreover, the graph will be embedded in $\SS^2$
	throughout the algorithm.

	We start with an embedding of $\theta^+$ in $\SS^2$, such that the two
	outerplanar components are embedded as two disjoint polygons with
	non-crossing diagonals.  Interiors of the two polygons are declared as
	forbidden area: during our construction, no new edge is allowed to intersect
	this area.  In other words, we are only allowed to draw edges within a belt.

	For the first step, we draw a curve connecting $1^+$ and $1^-$, to which we
	are free to assign any non-positive weight ranging from $-\min(w_{1^+},
	w_{1^-})$ to $0$.  We also have the freedom to choose a sign $\sigma=\pm$,
	and increment $i$ if $\sigma=+$, increment $j$ if $\sigma=-$.

	In the following steps, we adopt a greedy strategy.  We draw a curve
	connecting $i^+$ and $j^-$, to which we assign the weight $-\min(w_{i^+},
	w_{j^-})$.  After this assignment, the face bounded by $i^+j^-$ and the
	previously assigned edge is considered as a forbidden area for later
	construction: no new edge is allowed to intersect this area.  Then we
	increment $i$ if $w_{i^+}=0$, and increment $j$ if $w_{j^-}=0$.  If both
	weights vanish, we increment both indices.

	Eventually, we will have $i=p+1$ and $j=q+1$, and the weights vanish at all
	vertices.  Then the algorithm stops.  The result is by construction the
	embedding of a $(p,q)$-admissible graph.

	In this algorithm, being greedy is not only good, but also necessary.  Note
	that the weights between vertices of smaller indices are already fixed.  If
	we choose any bigger negative weight for the edge $i^+j^-$, then both
	$w_{i^+}$ and $w_{j^-}$ remain positive.  They both need to be connected to a
	vertex with larger index to cancel the weight.  This is however not possible,
	since neighborhoods of these vertices have been declared as forbidden area.

	The algorithm is parametrized only by the two choices at the first step: a
	non-positive weight and a sign.  The space of parameters is therefore
	homeomorphic to a segment.
\end{proof}

We have $\Theta^{-1}(\cA_{p,q}^i) \subseteq \cP_{p,q}^i$.  Let $\Theta_i$
denote the restriction of $\Theta$ on $\cP_{p,q}^i$; it is a covering map with
images in $\cA_{p,q}^i$.  Since $\cP_{p,q}^i$ is connected and $\cA_{p,q}^i$
simply connected, we conclude that $\Theta_{p,q}^i$ is a homeomorphism.  This
proves that the covering number of $\Theta$ is $1$, i.e.\ $\Theta$ is a
homeomorphism.

\subsection{Admissible HS structures}

%
Let $\HHP^2_\ell$ denote the complete, simply connected hyperbolic surface with
one cone singularity of angle $\ell$.  We extend this notation, and use
$\HHP^2_0$ for the hyperbolic surface with one cusp.  A non-degenerate boundary
component of $\dS^2_\ell$ can be identified to the boundary of $\HHP^2_\ell$.

Let $B$ be a subset of $p$ points on $\partial\HHP^2$, considered up to
isometries of $\HHP^2$.  Let $\cG_{B,\ell}$ be the space of CPP maps from
$\partial\HHP^2$ to $\partial\HHP^2_\ell$, $0 < \ell < 2\pi$, up to isometries
of both $\HHP^2$ and $\HHP^2_\ell$, with positive break point set $B$.  We
denote by $\cG_B$ the union of $\cG_{B,\ell}$ for $0<\ell<2\pi$, i.e.\ set of
all CPP maps on $\partial\HHP^2$ with positive break point set $B$, up to
isometries.

A \emph{horocyclic $p$-gon} is the intersection of $p$ horodisks in $\HHP^2$.
Figure~\ref{fig:horogon} shows a horocyclic $3$-gon and a horocyclic $4$-gon. 

\begin{figure}[hbt] 
  \centering 
  \includegraphics[width=.9\textwidth]{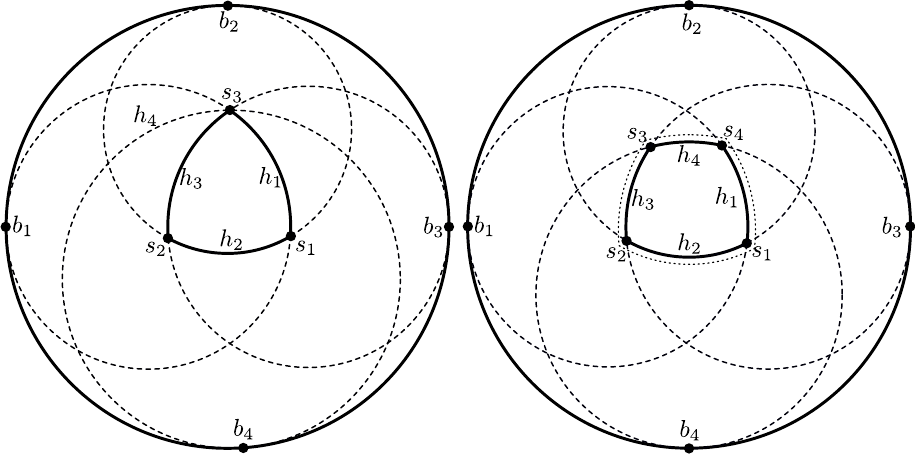}
 	\caption{
 		The horocyclic $3$-gon on the left, together with the horocycle $h_4$,
 		represent a point in $\partial\cG_B$ with $B=\{b_1,b_2,b_3,b_4\}$.
 		Shrinking $h_4$ truncates the $3$-gon into the $4$-gon on the right.  We
 		also sketch a scaling of the $4$-gon, which is the key in the proof of
 		Proposition~\ref{prop:2p-4}.  \label{fig:horogon}
	}
\end{figure}

The key observation is the following homeomorphism $\eta$ from $\cG_B$ to the
space of horocyclic $p$-gons bounded by horocycles based at $B$.

Label the elements of $B$ as $b_1, \dots, b_p=b_0$ in the clockwise order in
the disk model.  They are the vertices of an ideal $p$-gon $P \subset \HHP^2$.
Consider a map $\gamma \in \cG_{B,\ell}$.  It maps $P$ to an ideal $p$-gon $P'
\subset \HHP^2_\ell$ with a cone singularity of angle $\ell$ in its interior.
The vertices of $P'$ are $\gamma(b_i)$.

We then obtain a horocyclic $p$-gon $\eta(\gamma)$ as follows.  Cut $P'$ into
$p$ triangles by the geodesics from $\gamma(b_i)$ to $s$, the cone singularity
of $\HHP^2_\ell$.  Then fold each triangle $\gamma(b_i) s \gamma(b_{i+1})$
inward, isometrically, into the triangle $b_i s_i b_{i+1}$ in $\HHP^2$.  Let
$h_i$ be the horocycle based at $b_i$ passing through $s_i$. Then
$\eta(\gamma)$ is the intersection of the horodisks bounded by the $h_i$'s.

Since $\gamma$ preserves horocycles, $s_{i-1}$ must also lie on $h_i$.  Hence
the horocyclic segments $s_is_{i+1}$ form a piecewise horocyclic closed curve,
denoted by $h$.

\begin{lemma}
  $h$ is embedded as the boundary of $\eta(\gamma)$.
\end{lemma}

\begin{proof}
	Since $b_i$ are positive break points of $\gamma$, the triangles $b_i s_i
	b_{i+1}$ and $b_{i-1}s_{i-1}b_i$ must overlap; see Proposition~\ref{prop:C2}
	and Remark~\ref{rem:martin}.  In other words, the points $b_i$, $s_{i-1}$ and
	$s_i$ are placed on $h_i$ in the clockwise order.  See
	Figure~\ref{fig:horogon} for examples.

	For $x \in \partial\HHP^2$ and $y \in \HHP^2$, we use $\xi(x,y)$ to denote the
	other ideal end of the geodesic that emerges from $x$ and passes through $y$.
	Define a map $g:\partial\HHP^2\to\partial\HHP^2$ such that $g^{-1}(b_i) =
	\{\xi(b_i, y) \mid y \in h_i \cap h\}$ and, for some $x$ between $b_i$ and
	$b_{i+1}$, $g^{-1}(x) = \xi(x, s_i)$.  The map $g$ is continuous and
	monotone, and has the property that $x \notin g^{-1}(x)$, hence its degree
	must be $1$.  This proves that $h$ is embedded, hence the boundary of
	$\eta(\gamma)$.
\end{proof}

Conversely, given a horocyclic polygon bounded by horocycles $h_i$ based at
$b_i \in B$, let $s_i \in h_i \cap h_{i+1}$ be its vertices.  Then we can glue
the triangles $b_i s_i b_{i+1}$ into an ideal $p$-gon with a cone singularity
of angle $\ell$.  More specifically, $\ell$ is the sum of the angles $\angle
b_i s_i b_{i+1}$.  This proves that $\eta$ is a homeomorphism.

\begin{remark}
	It is interesting to note from the horocyclic polygons that $\ell < 2\pi$.
	Let $s$ be a point in the interior of the horocyclic polygon.  We then have
	$\angle b_i s_i b_{i+1} < \angle b_i s b_{i+1}$ for all $i$.  Yet the sum of
	$\angle b_i s b_{i+1}$ is equal to $2\pi$.
\end{remark}

\begin{proposition}
  $\cG_B$ is homeomorphic to $\RR^{|B|}$.
\end{proposition}

\begin{proof}
  The proof is by induction on the cardinality $p=|B|$.  Up to isometries of
  $\HHP^2$, we may consider $B$ fixed.

  For $p=2$, a horocyclic $2$-gon $P_2$ is bounded by two horocycles.  It is
  determined by the position of an intersection point of the two horocycles.
  This point can be chosen arbitrarily in $\HHP^2$, proving the statement for
  $p=2$.

  Now consider a horocyclic $(p-1)$-gon $P_{p-1}$ bounded by horocycles $h_1,
  \dots, h_{p-1}$ based at $b_1, \dots, b_{p-1} \in B$, and let $s_i \in h_i
  \cap h_{i+1}$ and $s_{p-1} \in h_{p-1} \cap h_1$ denote the vertices of
  $P_{p-1}$.  We now construct a horocyclic $p$-gon $P$ bounded by horocycles
  with bases in $B$.  For this, it suffices to choose a horocycle $h_p$ that
  truncates the vertex $s_{p-1}$ of $P_{p-1}$.  This can be done by taking the
  horocycle $h_p$ at $b_p$ passing through $s_{p-1}$, then shrinking it.  On
  the left of Figure~\ref{fig:horogon} we illustrate a truncation of a $3$-gon
  into a $4$-gon.  We can continue to shrink $h_p$ until it hits another vertex
  of $P_{p-1}$.

  Hence $\cG_B$ is homeomorphic to $\cG_{B \setminus \{b_p\}} \times \RR$,
  which is $\RR^{|B|}$ by induction.
\end{proof}

In the following we will consider the closures of $\cG_B$.  The boundary
$\partial\cG_B$ consists of three parts, namely $\cG_{B,0}$, $\cG_{B,2\pi}$,
and $\partial\cG_{B,\ell}$, $0 < \ell < 2\pi$.  We now define and describe
these spaces.

We use the notation $\cG_{B,0}$ for the space of CPP maps from $\partial\HHP^2$
to $\partial\HHP^2_0$, up to isometries of both $\HHP^2$ and $\HHP^2_0$, with
positive break point set $B$.  As before, we can interpret a map $\gamma \in
\cG_{p,0}$ as gluing the boundary of an ideal $p$-gon $P \subset \HHP^2$ to the
boundary of an ideal $p$-gon $P' \subset \HHP^2_0$, where $P'$ contains a cusp
$s$.  We triangulate $P'$ by connecting its vertices to $s$, and triangulate
$P$ arbitrarily, and glue them through $\gamma$ to obtain a triangulation $T$
of the $2$-sphere.  The shearing coordinates on the $p-3$ interior edges of $P$
are determined by the positions of the break points.  The shearing coordinates
on the remaining $2p$ edges of $T$ are governed by the condition that the
shearing around each vertex of $T$ should sum up to $0$.  Since $T$ has $p+1$
vertices, we conclude that $\cG_{B,0}$ is homeomorphic to
$\RR^{p-1}=\RR^{|B|-1}$.

The part $\partial\cG_{B,\ell}$, $0 < \ell < 2\pi$, consists of CPP maps from
$\partial\HHP^2$ to $\partial\HHP^2_\ell$ up to isometries of both $\HHP^2$ and
$\HHP^2_\ell$, with break point set $B' \subset B$ and marked projective point
set $B \setminus B'$.  The homeomorphism $\eta$ extends continuously to this
boundary.  More specifically, let $p=|B|$ and $p'=|B'|$.  Then for $\gamma \in
\partial\cG_{B,\ell}$, $\eta(\gamma)$ is a horocyclic $p'$-gons $P'$ bounded by
horocycles with bases in $B'$, together with $p-p'$ additional horocycles with
bases in $B\setminus B'$ ``supporting'' $P'$ (that is, they intersect $\partial
P'$ but disjoint from the interior of $P'$).  The left side of
Figure~\ref{fig:horogon} is an example with $p=4$ and $p'=3$.

We also abuse the notation $\cG_{p,2\pi}$ for the space of projective
homeomorphisms from $\partial\HHP^2$ to itself up to isometries of $\HHP^2$,
with a set of marked points $B$.  In fact, the marking here is superficial;
hence $\cG_{p,2\pi}$ consists of a single element.

\medskip

Let $\delta_i$ denote the distance from $s_i$ to the geodesic $b_ib_{i+1}$.
The hyperbolic triangle formula shows that $\alpha_i$ and $\delta_i$ are
related by the formula $\cosh\delta_i \sin\alpha_i/2 = 1$.  We now deform the
horocyclic $p$-gon by moving every $s_i$, simultaneously, along the geodesic
perpendicular to $b_ib_{i+1}$, to a new position $s'_i$.  The following lemma
is the crucial observation for the proof that follows later.

\begin{lemma} \label{lem:deform}
	If the deformation is performed in such a way that the ratio
	$\cosh\delta_i/\cosh\delta'_i$ is the same for every $i$, then there is a
	horocycle $h'_i$ passing through every adjacent pair $s'_{i-1}$ and $s'_i$.
\end{lemma}

\begin{proof}
	Let $k$ be the common ratio of $\cosh\delta_i / \cosh\delta'_i$.  We use the
	half-plane model of $\HHP^2$, and assume that $b_i=\infty$.  The situation is
	illustrated in Figure~\ref{fig:deform}.  Let $\psi_i$ be the (Euclidean)
	angle $\angle s_i b_{i+1} b_{i-1}$ and $\psi_{i-1}$ be the angle $\angle
	s_{i-1} b_{i-1} b_{i+1}$.  Then for $j=i$ or $i-1$, we can calculate the
	distances (see for instance \cite[\S 3.5]{anderson2005})
	\[
		\delta_j = \ln\left| \operatorname{csc} \psi_j + \cot \psi_j \right|
		= \operatorname{arccosh}\operatorname{csc}\psi_j.
	\]
	This is particularly convenient because we have $\cosh\delta_j \sin\psi_j =
	1$.  Hence $\cosh\delta_j / \cosh\delta'_j=k$ implies that $\sin\psi_j /
	\sin\psi'_j = 1/k$ for every $i$.  In the half-plane model, $s_i$ and
	$s_{i-1}$ are moving along the circles centered at $b_{i+1}$ and $b_{i-1}$,
	respectively, such that their heights are both scaled by $1/k$.  Therefore,
	their new positions $s'_{i-1}$ and $s'_i$ are again at the same height, hence
	on the same horocycle $h'_i$ based at $b_i = \infty$.
\end{proof}

\begin{figure}[hbt] 
  \centering 
  \includegraphics[width=.4\textwidth]{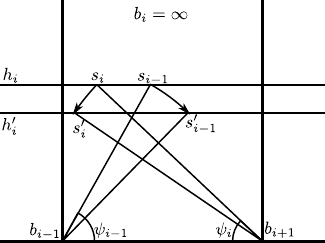}
 	\caption{
 		Proof of Lemma~\ref{lem:deform}. \label{fig:deform}
	}
\end{figure}

This deformation is sketched on the right side of Figure~\ref{fig:horogon}.  We
see from Figure~\ref{fig:deform} that, by moving $s_j$'s outwards the
horocyclic polygon, one can multiply $\cosh(\delta_i)$ by an arbitrarily large
constant.  However, if we move $s_j$'s inwards the horocyclic polygon, $s_i$
and $s_{i-1}$ would eventually merge.

\begin{proposition} \label{prop:2p-4}
	$\cG_{B,\ell}$, $0 \le \ell \le 2\pi$, are contractible, and homeomorphic to
	$\RR^{|B|-1}$ if $\ell < 2\pi$.
\end{proposition}

\begin{proof}
  Through any given point $\gamma \in \cG_B$, the deformation described above
  defines a continuous path with monotonically changing cone angle $\ell$.  We
  use $\ell$ as the parameter and denote this path by $c_\gamma(\ell)$.  See
  Figure~\ref{fig:gpl} for an illustration.

  To decrease $\ell$, one needs to move the vertices outwards the horocyclic
  polygon.  We have seen that, in this direction, one can travel along
  $c_\gamma(\ell)$ until hitting $\cG_{B,0}$.

  To increase $\ell$, one needs to move the vertices inwards the horocyclic
  polygon.  In this direction, however, the path $c_\gamma(\ell)$ would in
  general hit some $\gamma' \in \partial\cG_{B,\ell}$ for some $\ell < 2\pi$,
  as shown in Figure~\ref{fig:gpl}.  Further movement of the vertices in the
  same direction would destroy the $p$-gon.  However, we can continue deforming
  $\gamma'$ as a gluing map in $\cG_{B'}$ for some $B'\subset B$.  Hence
  $c_\gamma(\ell)$ is extended into the closure of $\cG_B$, along which one can
  increase $\ell$ up to $2\pi$; see Figure~\ref{fig:gpl}.

	This path is uniquely defined through every $\gamma \in \overline\cG_B$, and
	two path do not intersect inside $\cG_B$; intersection is only possible on
	the boundary.  For $0 \le \ell \ne \ell' \le 2\pi$, let $f_{\ell,\ell'}
	\colon \cG_{B,\ell} \to \cG_{B,\ell'}$ be the continuous map that sends
	$\gamma \in \cG_{B,\ell}$ to $\gamma' = c_\gamma(\ell') \in \cG_{B,\ell'}$.
	Then $f_{\ell, \ell'}$ and $f_{\ell',\ell}$ define a homotopy equivalence
	between $\overline\cG_{B,\ell}$ and $\overline\cG_{B,\ell'}$.  Consequently,
	$\overline\cG_{B,\ell}$ are all of the same homotopy type.  We have seen that
	$\cG_{B,0}$ and $\cG_{B,2\pi}$ are contractible, and so must be
	$\cG_{B,\ell}$ for $0 < \ell < 2\pi$.

	In general, $f_{\ell,\ell'}$ and $f_{\ell',\ell}$ are not inverse to each
	other.  However, if the path emerging from $\gamma \in \cG_{B,\ell}$ arrives
	at $\gamma' = f_{\ell,\ell'}(\gamma) \in \cG_{B,\ell'}$ without hitting the
	boundary of $\cG_B$, then one can travel backwards along the same path.  The
	reversed path defines $f_{\ell',\ell}$, hence we have $f_{\ell',\ell} \circ
	f_{\ell,\ell'}(\gamma) = f_{\ell',\ell}(\gamma') = \gamma$.

	This is the case, in particular, when $\ell > \ell'$ and $\gamma \in
	\cG_{B,\ell}$.  Then $f_{\ell,\ell'}$ defines a homeomorphism between
	$\cG_{B,\ell}$ and its image $f_{\ell,\ell'}(\cG_{B,\ell}) \subset
	\cG_{B,\ell'}$.  We finally conclude that $\cG_{B,\ell}$, $0<\ell<2\pi$, by
	its continuity in $\ell$, are all homeomorphic to $\cG_{B,0}$, thus to
	$\RR^{|B|-1}$.
\end{proof}

It is now clear that $\cG_B$ is foliated by $\cG_{B,\ell}$, $0 < \ell < 2\pi$,
as illustrated in Figure~\ref{fig:gpl}.

\begin{figure}[hbt]
  \centering 
  \includegraphics[width=.3\textwidth]{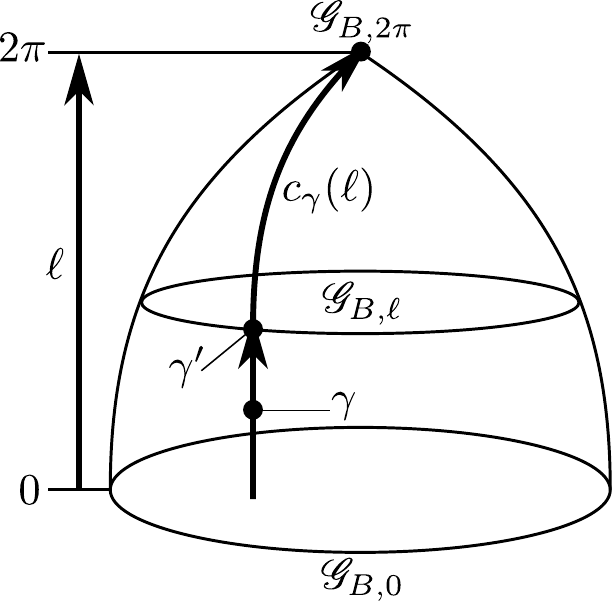}
 	\caption{
 		The structure of $\cG_B$ showing a path $c_\gamma$.  \label{fig:gpl}
	}
\end{figure}

An admissible HS structure is obtained by gluing copies of $\HHP^2$ to the
ideal boundary components of $\dS^2_\ell$.  We now conclude on the topology of
$\cM_{p,q}$, and prove Theorem~\ref{thm:metric}.

If $p < 2 < q$, an element of $\cM_{1,q}$ can be constructed by first choosing
a set of $q$ points $B \subset \partial\HHP^2$ up to isometries, and then a
gluing map $\gamma \in \cG_{B,0}$.  For the first step, we may fix three points
up to hyperbolic isometry, hence the space of $B$ is homeomorphic $\RR^{q-3}$.
In the second step, we have seen that $\cG_{B,0}$ is homeomorphic to
$\RR^{q-1}$.  Hence $\cM_{1,q}$ is homeomorphic to $\RR^{2q-4} =
\RR^{2(p+q-3)}$.  Theorem~\ref{thm:metric} follows since both $\cG_{p,0}$ and
$\cP_{p,0}$ are contractible.

If $2 \le p \le q$, we first count the dimension.  For the gluing map on one
boundary of $\dS_\ell$, we need to choose a set $B$ of $p$ break points, then
pick a gluing map from $\cG_B$.  Up to isometries, the space of this gluing map
is of dimension $2p-3$.  Similarly, the gluing map on the other boundary
contributes $2q-3$ dimensions.  The two gluing maps have the same cone angle,
removing one degree of freedom.  But we can also rotate the break points on
$\partial\HHP^2$, corresponding to translations in $\RRP^1$.  This contributes
another dimension, hence the dimension of $\cM_{p,q}$ is $2(p+q-3)$.

The rotation of $\partial\HHP^2$ generates the non-trivial fundamental group of
$\cM_{p,q}$.  To prove that the map $\Delta$ is a homeomorphism, we lift it to
a map $\tilde{\Delta}$ between the universal covers $\tilde\cP_{p,q}$ and
$\tilde\cM_{p,q}$.  A point in $\tilde\cP_{p,q}$ corresponds to a $(p,q)$-ideal
polyhedron equipped with a path (defined up to homotopy) connecting vertex
$1^+$ and $1^-$.  A point in $\tilde\cM_{p,q}$ corresponds to a
$(p,q)$-admissible HS structure with a path (up to homotopy) connecting $1^+$
and $1^-$ in $\dS^2_\ell$.  Hence $\tilde\Delta$ is a homeomorphism between
marked $(p,q)$-ideal polyhedra and marked $(p,q)$-admissible HS structures.
This proves that the covering number of $\Delta$ is $1$.

\section{Combinatorics}\label{sec:combinatorics}
It remains to prove Theorem \ref{thm:combinatorics} from
Theorem~\ref{thm:weights}.  In other terms, assume that a graph $\Gamma=(V,E)$
satisfies Condition~\ref{con:2cycles} and the edges are colored as specified in
Theorem~\ref{thm:weights}.  We need to prove that Condition~\ref{con:alternate}
implies the existence of a weight function $w:E \to \RR$ satisfying
Conditions~\ref{con:sign} and~\ref{con:wsum} and, conversely, existence of such
a weight function implies Condition~\ref{con:alternate}.

We consider two cases.


\paragraph*{Case $p<2<q$}

In this case, the cycle cover contains a $1$-cycle, say with vertex set
$V^+=\{v_0\}$.  The other vertices $V^-$ induce a 2-connected outerplanar
subgraph.  We color the edges adjacent to $v_0$ by blue, and other edges by
red.

Theorem~\ref{thm:combinatorics} requires a cycle visiting all the edges along
which the edge color has the pattern~\dots-blue-blue-red-\dots.  This is
actually equivalent to a much simpler condition:

\begin{lemma}
	In the case of $p<2<q$, Condition~\ref{con:alternate} is equivalent to
	\renewcommand{\theenumi}{\rm\bf(C'\arabic{enumi})}
	\renewcommand{\labelenumi}{\theenumi}
  \begin{enumerate}
  		\setcounter{enumi}{1}
  	\item \label{con:dominate} $v_0$ is connected to every vertex in $V^-$.
  \end{enumerate}
\end{lemma}

\begin{proof}\leavevmode
	\begin{description}
		\item[\ref{con:dominate} $\implies$ \ref{con:alternate}] Immediate.

		\item[\ref{con:alternate} $\implies$ \ref{con:dominate}] If some $v \in
			V^-$ is not connected to $v_0$, then the edges adjacent to $v$ can not
			belong to a cycle as required in Condition~\ref{con:alternate}.
	\end{description}
\end{proof}

\begin{proof}[Proof of Theorem~\ref{thm:combinatorics} when $p<2<q$]\leavevmode
	\begin{description}
		\item[\ref{con:alternate} $\implies$ \ref{con:sign} $\wedge$ \ref{con:wsum}]
			Suppose a cycle $c$ specified by Condition~\ref{con:alternate}.  Let $n$
			be its length; $n$ is necessarily a multiple of $3$.  Let $n_e$ be the
			number of times that $c$ visits $e$.  Assign to $e$ the weight $n_e$ if
			$e$ is red, or the weight $-n_e$ if $e$ is blue.  After normalization by
			a factor $3\pi/n$, we obtain a graph that satisfies
			Conditions~\ref{con:sign} and~\ref{con:wsum}.

		\item[\ref{con:sign} $\wedge$ \ref{con:wsum} $\implies$ \ref{con:dominate}]
			Assume a graph function satisfying Condition~\ref{con:sign}.  If some
			vertex $v \in V^-$ is not connected to $v_0$, then the edges adjacent to
			$v$ are all of positive weight.  This violates Condition~\ref{con:wsum}.
	\end{description}
\end{proof}

\paragraph*{Case $2 \le p \le q$}

Condition~\ref{con:alternate} requires a cycle with alternating colors; we call
such a cycle an \emph{alternating cycle}.  As in the previous case, the
existence of an alternating cycle that visits every edge implies that every
vertex is adjacent to a blue edge, but the converse is not true.  However, we
have the following equivalence, which does not depend on
Condition~\ref{con:2cycles}.

\begin{proposition} \label{pr:combi}
	If the edges of a graph are colored in blue and red, then
	Condition~\ref{con:alternate} is equivalent to
	\renewcommand{\theenumi}{\rm\bf(C''\arabic{enumi})}
	\renewcommand{\labelenumi}{\theenumi}
  \begin{enumerate}
  		\setcounter{enumi}{1}
  	\item \label{con:cyclebase} Every edge belongs to an alternating cycle (which does not
  		necessarily visits every edge).
  \end{enumerate}
\end{proposition}

The proof is immediate from the composition and decomposition of cycles.

\begin{proof}[Proof of Theorem~\ref{thm:combinatorics} when $2 \le p \le q$]\leavevmode
	\begin{description}
		\item[\ref{con:alternate} $\implies$ \ref{con:sign} $\wedge$
			\ref{con:wsum}] Suppose a cycle $c$ specified by
			Condition~\ref{con:alternate}.  Let $n_e$ be the number of times that $c$
			visits $e$.  Assign to $e$ the weight $n_e$ if $e$ is red, or the weight
			$-n_e$ if $e$ is blue.  We obtain a graph that satisfies
			Conditions~\ref{con:sign} and~\ref{con:wsum}.

		\item[\ref{con:sign} $\wedge$ \ref{con:wsum} $\implies$
			\ref{con:cyclebase}] Let $w$ be a graph satisfying
			Conditions~\ref{con:sign} and~\ref{con:wsum}.
			
			If $w$ has an alternating cycle $c$, the number of visits defines a graph
			supported on the edges of $c$, which we denote by $w_c$.  We can choose a
			positive number $\alpha$ such that $w' = w - \alpha w_c$ is supported on
			a proper subgraph of $w$, but still satisfy Condition~\ref{con:sign} on
			other edges.  Most importantly, $w'$ satisfy Condition~\ref{con:wsum}
			because both $w$ and $w_c$ do.  The cycle $c$ no longer exist in $w'$.
			We repeat this operation if there are other alternating cycles.  Since
			the graph is finite, we will obtain a graph $\tilde w$ with no
			alternating cycle in finitely many steps.

			If an edge $e_0$ of $w$ does not belong to any alternating cycle, it must
			also be an edge of $\tilde w$.  But we prove in the following that this
			is not possible.
			
			Assume that $e_0$ is red and let $v_0^+$ and $v_0^-$ be its vertices.
			Condition~\ref{con:wsum} guarantees that $v_0^+$ is adjacent to a blue
			edge, $e_1$, whose other vertex is denoted by $v_1^+$.  The same argument
			continues and we obtain an alternating path $e_0, e_1, e_2, \dots$.
			
			Since the graph is finite, this path will eventually intersect itself.
			That is, $v_i^+ = v_{i'}^+$ for some $0 \le i'<i$ (note that we don't
			consider $v_0^-$ as visited by $e_0$).  If $e_i$ and $e_{i'}$ are of the
			same color, $e_{i'+1}, \dots, e_i$ form an alternating cycle,
			contradicting our assumption.  Hence $e_i$ and $e_{i'}$ must have
			different colors.

			The same argument applies on the other vertex $v_0^-$ of $e_0$.  We
			obtain an alternating path $e_{0}, e_{-1}, e_{-2}, \dots$.  Let $v_j^-$
			denote the common vertex of $e_{-j}$ and $e_{-j-1}$.  This path
			eventually intersect itself, i.e. $v_j^- = v_{j'}^-$ for some $0 \le j' <
			j$ (this time we don't consider $v_0^+$ as visited by $e_0$).  Again,
			$e_{-j}$ and $e_{-j'}$ must have different colors.

			But then, $e_0, \dots, e_i, e_{i'}, \dots, e_0, \dots, e_{-j}, e_{-j'},
			\dots, e_0$ form an alternating cycle; see Figure~\ref{fig:ProofComb}.
			This contradicts our assumption.
	\end{description}
\end{proof}

\begin{figure}[hbt] 
  \centering 
  \includegraphics[width=.7\textwidth]{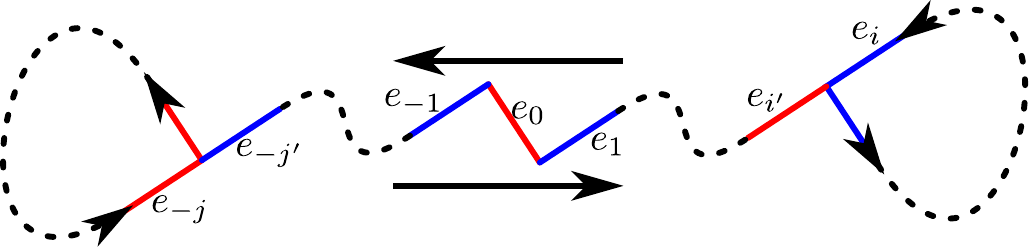}
 	\caption{
 		The alternating cycle in the last step in the proof of
 		Theorem~\ref{thm:combinatorics}.  \label{fig:ProofComb}
	}
\end{figure}

\begin{remark}
	The feasibility region specified in Theorem~\ref{thm:weights} is a polyhedral
	cone.  The proof above shows that the extreme rays of this cone correspond to
	the minimal alternating cycles in the graph.
\end{remark}

\begin{figure}[hbt] 
  \centering 
  \includegraphics[height=4cm]{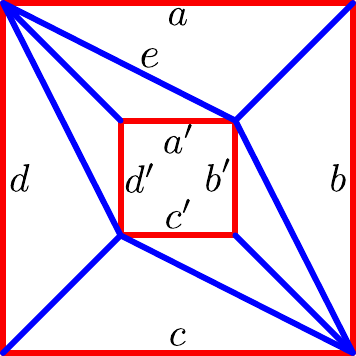}
 	\caption{
  	This graph is not the $1$-skeleton of any weakly ideal polyhedron.
  	\label{fig:Example}
	}
\end{figure}

\begin{example}
	The example in Figure \ref{fig:Example} shows that
	Condition~\ref{con:alternate} is essential.  This graph is not the
	$1$-skeleton of any weakly ideal polyhedron with the inner square in
	$\HHP^3_+$ and the outer square in $\HHP^3_-$.  A fairly elementary argument,
	left to the reader, shows that there is no alternating cycle containing edge
	$e$.  This can also be shown using Theorem~\ref{thm:weights}, since if a
	graph $w$ satisfies Conditions~\ref{con:sign} and~\ref{con:wsum}, we would
	have
	\begin{align*}
		w(a)+w(b) & < w(a')+w(b')~,\\
		w(c)+w(d) & < w(c')+w(d')~,\\
		w(b)+w(c) & > w(b')+w(c')~,\\
		w(d)+w(a) & > w(d')+w(a')~,
	\end{align*}
	from which a contradiction immediately follows.
\end{example}

\begin{remark}
	Given a graph $G$ with edges colored in blue and red, we define a directed graph
	$\tilde G$ as follows:
	\begin{itemize}
		\item Each vertex $v$ of $G$ lifts to two vertices $v_+$ and $v_-$ in
			$\tilde G$.

		\item Each red edge $uv$ of $G$ lifts to two oriented edges $u_-v_+$ and
			$v_-u_+$ in $\tilde G$.

		\item Each blue edge $uv$ of $G$ lifts to two oriented edges $u_+v_-$ and
			$v_+u_-$ in $\tilde G$.
	\end{itemize}
	It is quite clear from the definition that an alternating cycle in $G$ lifts
	to two oriented cycle in $\tilde G$, and any oriented cycle in $\tilde G$
	projects to an alternating cycle in $G$.  Hence finding an alternating cycle
	in $G$ is equivalent to finding an oriented cycle in $\tilde G$.  The latter
	can be solved by a simple depth- or breath-first search.
\end{remark}

\bibliographystyle{alpha}
\bibliography{References}

\end{document}